\documentclass[a4paper,10pt]{article}

\usepackage{amsmath}
\usepackage{amsthm}
\usepackage{setspace}
\usepackage{graphicx}

\usepackage{amsfonts}
\usepackage{adjustbox}
\usepackage{natbib}
\usepackage[affil-it]{authblk}
\usepackage{xr}
\usepackage{hyperref}
\usepackage{enumitem}

\newcommand{\vect}[1]{\ensuremath{\boldsymbol{#1}}}
\newcommand{\mat}[1]{\ensuremath{\boldsymbol{#1}}}

\newtheorem{definition}{Definition}
\newtheorem{proposition}{Proposition}

\newtheorem{remark}{Remark}

\title{Robust Multivariate Estimation Based On Statistical Depth Filters}
\author[1]{Giovanni Saraceno, \thanks{Dipartimento di Matematica, Universit\`a degli studi di Trento, Via Sommarive 14, 38123 Povo(Trento), Italy. \\ Tel.: +39-3208918364
  \\ email: \texttt{giovanni.saraceno@unitn.it}}}
\author[1]{Claudio Agostinelli}
\affil[1]{Dipartimento di Matematica, Universit\`a degli studi di Trento}

\begin{document}

\maketitle

\begin{abstract}
In the classical contamination models, such as the gross-error (Huber and Tukey contamination model or Case-wise Contamination), observations are considered as the units to be identified as outliers or not. This model is very useful when the number of considered variables is moderately small. \citet{Alqallaf2009} shows the limits of this approach for a larger number of variables and introduced the Independent contamination model (Cell-wise Contamination) where now the cells are the units to be identified as outliers or not. One approach to deal, at the same time, with both type of contamination is filter out the contaminated cells from the data set and then apply a robust procedure able to handle case-wise outliers and missing values. Here we develop a general framework to build filters in any dimension based on statistical data depth functions. We show that previous approaches, e.g. \citet{Agostinelli2015a} and \citet{Zamar2017}, are special cases. We illustrate our method by using the half-space depth.


\smallskip


\end{abstract}

\section{Introduction}
\label{sec:introduction}

One of most common problem in real data is the presence of outliers, i.e. observations that are well separated from the bulk of data, that may be errors that affect the data analysis or can suggest unexpected information. According to the classical Tukey-Huber Contamination Model (THCM), a small fraction of rows can be contaminated and these are the units considered as outliers. Since the $1960$'s many methods have been developed in order to be less sensitive to such outlying observations. A complete introduction and explanation of the developments in robust statistics is given in the book by \citet{Maronna2006}.

In some application, e.g. in modern high-dimensional data sets, the entries of an observation (or cells) can be independently contaminated. \citet{Alqallaf2009} first formulated the Independent Contamination Model (ICM), taking into consideration this cell-wise contamination scheme. According to this paradigm, given a fraction $\epsilon$ of contaminated cells, the expected fraction of contaminated rows is
\begin{equation*}
  1 - (1-\epsilon)^p
\end{equation*}
which exceeds the $50\%$ breakdown point for increasing value of the contamination level $\epsilon$ and the dimension $p$. Traditional robust estimators may fail in this situation. Furthermore, \citet{Agostinelli2015b} showed that both type of outliers, case-wise and cell-wise, can occur simultaneously.

\citet{Gervini2002} introduced the idea of an adaptive univariate filter, identifying the proportion of outliers in the sample measuring the difference between the empirical distribution and a reference distribution. Then, it is used to compute an adaptive cutoff value, and finally a robust and efficient weighted least squares estimator is defined. Starting from this concept of outlier detection, \citet{Agostinelli2015a} introduced a two-step procedure: in the first step large cell-wise outliers are flagged by the univariate filter and replaced by NA's values \citep[a technique called snipping in][]{Farcomeni2014}; in the second step a Generalized S-Estimator \citep{Danilov2012} is applied to deal with case-wise outliers. The choice of using GSE is due to the fact that it has been specifically designed to cope with missing values in multivariate data. \citet{Zamar2017} improved this procedure proposing the following modifications:
\begin{itemize}
\item they combined the univariate filter with a bivariate filter to take into account the correlations among variables;
\item in order to handle also moderate cell-wise outliers, they proposed a filter as intersection between the univariate-bivariate filter and \textit{Detect Deviating Cells} (DDC), a filter procedure introduced by \citet{Rousseeuw2018};
\item finally, they constructed a Generalized Rocke S-estimator (GRE) replacing the GSE, to face the lost of robustness in case of high-dimensional case-wise outliers.
\end{itemize}

Here, we introduce a general idea of constructing filters in general dimension $d$, with $1 \le d \le p$, based on the statistical data depth functions, namely depth-filters. In particular, we show that the previously mentioned univariate-bivariate filter is a special case, if an appropriate statistical depth function is used.

  We develop one of these depth-filters using the half-space depth, HS-filter. Thus, we repropose the two steps procedure. In the first step, we apply the HS-filter taking $d=1$, $d=2$ and $d=p$, in sequence. As in \citet{Zamar2017}, the univariate and bivariate filters are combined in order to identify outlying cells which are replaced by NA's values. Note that, if $d = 1$, we filter the cell-wise outliers considering the variables as independent. Finally, the HS-filter with $d=p$ is performed on observations, so that, we can find undetected case-wise outliers. In the second step, the Generalized S-estimator is used. Therefore, we also took into account the improvements suggested by \citet{Zamar2017}. Indeed, we improved our procedure following such modifications.

The rest of the work is organized as follows. Section \ref{sec:filters} introduces the main idea on how to construct filters based on statistical depth functions. In Section \ref{sec:gerviniyohai}, we show that the filters used in \citet{Agostinelli2015a} and \citet{Zamar2017}, namely GY-filters, are special cases of our proposed depth-filter approach, that is, they can be written in terms of depth functions. In order to prove that, we introduce a statistical data depth function called Gervini-Yohai depth function and we prove that the filter based on this depth coincides with the GY-filter. In Section \ref{sec:halfspace}, as an important example we consider the filter obtained by using the half-space depth function and in subsections \ref{sec:pvariate1} we introduce the proposed strategy to mark observations/cells as outliers. Section \ref{sec:simulation} reports the results of a Monte Carlo experiment while Section \ref{sec:example} illustrates the features of our approach using a simulation example and a real data set. Concluding remarks are given in Section \ref{sec:conclusions}. In the Supplementary Material, Section SM--1 discusses the general properties that a statistical data depth function should satisfy. The derivation of the claim in Remark \ref{rem:gy-filter-half} is provided in Section SM--2. In Section SM--3, we prove that the general properties introduced in SM--1 hold for the Gervini-Yohai depth. Section SM--4 illustrates the univariate HS-filter with two-tails control and Section SM--5 contains full results of the Monte Carlo experiment. Finally, Section SM--6 reports the codes used for the simulation example and for the real data set.

\section{Filters based on Statistical Data Depth Function}
\label{sec:filters}

Let $\vect{X}$ be a $\mathbb{R}^d$-valued random variable and $F$ a continuous distribution function. For a point $\vect{x} \in \mathbb{R}^d$, we consider the statistical data depth of $\vect{x}$ with respect to $F$ be $d(\vect{x};F)$, where $d(\cdot,F)$ satisfies the four properties given in \citet{Liu1990} and \citet{Zuo2000a} and reported in Section SM--1 of the Supplementary Material. Given an independent and identically distributed sample $\vect{X}_1, \ldots, \vect{X}_n$ of size $n$, we denote by $\hat{F}_n(\cdot)$ its empirical distribution function and by $d(\vect{x}; \hat{F}_n)$ the sample depth. We assume that, $d(\vect{x}; \hat{F}_n)$ is a uniform consistent estimator of $d(\vect{x}; F)$, that is, 
\begin{equation*}
\sup_{\vect{x}} | d(\vect{x}; \hat{F}_n) - d(\vect{x}; F) | \stackrel{a.s.}{\rightarrow} 0 \qquad n \rightarrow \infty ,
\end{equation*}
a property enjoined by many statistical data depth functions, e.g., among others simplicial depth \citep{Liu1990} and half-space depth \citep{Donoho1992}. One important feature of the depth functions is the $\alpha$-depth trimmed region given by
\begin{equation*}
  R_\alpha(F) = \{ \vect{x} \in \mathbb{R}^d: d(\vect{x}; F) \ge \alpha\}.
\end{equation*}
For any $\beta \in [0,1]$, $R^\beta(F)$ will denote the smallest region $R_\alpha(F)$ that has probability larger than or equal to $\beta$ according to $F$. Throughout, subscripts and superscripts for depth regions are used for depth levels and probability contents, respectively. Let $C^\beta(F)$ be the complement in $\mathbb{R}^d$ of the set $R^\beta(F)$. Let $m = \max_{\vect{x}} d(\vect{x};F)$ be the maximum value of the depth (for simplicial depth $m \le 2^{-p}$, for half-space depth $m \le 1/2$).

Given a high order probability $\beta$, we define a filter of dimension $d$ based on
\begin{equation}
\label{general_definition}
d_n = \sup_{\vect{x} \in C^\beta(F)} \{ d(\vect{x}; \hat{F}_n) - d(\vect{x}; F) \}^+ ,
\end{equation}
where $\{a\}^+$ represents the positive part of $a$. Then, we mark as outliers all the $n_0 = \left\lfloor \frac{n d_n}{2 m} \right\rfloor$ observations with the smallest population depth (where $\lfloor a \rfloor$ is the largest integer less then or equal to $a$). Given a depth function $d(\cdot,F)$, a desired property is that $\frac{n_0}{n} \to 0$ as $n \to \infty$. We recall the definition of consistent filter.

\begin{definition}
 Consider a random sample $\vect{X}_1,\ldots,\vect{X}_n$, where $\vect{X}_i$ are generated by the distribution $F_0$ and some cells can be independently contaminated. Let $\mathcal{F}$ be a filter, a procedure that flags some cells as cell-wise outliers replacing them by NA's, and let $d_n$ be the proportion of cells flagged by the filter. A filter is said \textit{consistent} for a given distribution $F_0$ if asymptotically it will not flag any cell if the data come from the true distribution $F_0$. That is
  \begin{equation*}
    \lim_{n \to \infty} d_n \rightarrow 0 \quad a.s. \ [F_0]
  \end{equation*}
\end{definition}

Note that, a statistical depth function can assume values in $\mathbb{R}^+ \cup \{0\}$. Hence, in order to be sure that the value $d_n$ is a proportion, we need to normalize this value dividing by the maximum $m$ of the depth. Intuitively, we can understand that the proportion of contaminated observations cannot exceed the 50\% since, in this case, it would not be possible to distinguish between the underlying distribution of data and the contaminating distribution. So, in addition, we divide also by 2, so that the final proportion of flagged observations as outliers lies between 0 and $1/2$.

\begin{remark}
  \label{rem:gy-filter-half}
  We verified that the filter proposed by \citet{Zamar2017} has a similar property. In particular, the probability that $d_n \ge \frac{1}{2}$ goes to 0 as $n \to \infty$. The derivation of this result is showed in Section SM--2 of Supplementary Material.
\end{remark}

\section{Gervini-Yohai d-variate filter}
\label{sec:gerviniyohai}

In this Section, we are going to show that the filters introduced in \cite{Agostinelli2015a} and \citet{Zamar2017} are a special case of our general approach to construct filters, that is, they can be expressed in terms of a depth function. For this reason, we are going to define a new depth, namely Gervini-Yohai depth, as follows
\begin{equation*}
d_{GY}(\vect{t}, F, G) = 1 - G(\Delta(\vect{t},\vect{\mu}(F),\mat{\Sigma}(F))) ,
\end{equation*}
where $G$ is a continuous distribution function, $\vect{\mu}(F)$ and $\mat{\Sigma}(F)$ are the location and scatter matrix functionals and $\Delta(t, F) = \Delta(\vect{t}, \vect{\mu}(F), \mat{\Sigma}(F)) = (\vect{t} - \vect{\mu}(F))^\top \mat{\Sigma}(F)^{-1} (\vect{t} - \vect{\mu}(F))$ indicates the squared Mahalanobis distance. In the Supplementary Material, Section SM--3 shows that this is a proper statistical data depth function since it satisfies the four properties that characterize a depth function.

Let $\{ G_n \}_{n=1}^\infty$ be a sequence of discrete distribution functions that might depends on $\hat{F}_n$ and such that
\begin{equation}
  \label{eqn:glivenko-cantelli}
  \sup_{t} |G_n(t) - G(t)| \stackrel{a.s.}{\rightarrow} 0.
\end{equation}
We might define the finite sample version of the Gervini-Yohai depth as
\begin{equation*}
d_{GY}(\vect{t}, \hat{F}_n, G_n) = 1 - G_n(\Delta(\vect{t},\vect{\mu}(\hat{F}_n),\mat{\Sigma}(\hat{F}_n))).
\end{equation*}
However, for filtering purpose we will use two alternative definitions later on.
The use of $G_n$, that might depend on the data, instead of $G$, makes this sample depth semiparametric.

Let $j_1, \ldots, j_d$, $1 \le d \le p$, be an $d$-tuple of the integer numbers in $\{1, \ldots, p\}$ and, for easy of presentation, let $\vect{Y}_i = (X_{ij_1}, \ldots , X_{ij_d})$ be a sub-vector of dimension $d$ of $\vect{X}_i$. Consider a pair of initial location and scatter estimators
\begin{equation*}
 \vect{T}_{0n}^{(d)} =  \left (
				\begin{array}{ll}
				T_{0n,j_1} \\
 				\ldots \\
				T_{0n,j_d}
				\end{array}
			\right )
			\quad \mbox{ and } \quad 
        \mat{C}_{0n}^{(d)} =  \left (
				\begin{array}{lll}
				C_{0n,j_1j_1} & \ldots   & C_{0n,j_1j_d} \\
 				\ldots             &\ldots    &   \ldots            \\
				C_{0n,j_dj_1}& \ldots    & C_{0n,j_dj_d}
				\end{array}
			\right ) \ .
\end{equation*}
Now, define the squared Mahalanobis distance for a data point $\vect{Y}_i$ by $\Delta_i = \Delta(\vect{Y}_i, \hat{F}_n) = \Delta(\vect{Y}_i, \vect{T}_{0n}^{(d)}, \mat{C}_{0n}^{(d)})$. Consider $G$ the distribution function of a $\chi_d^2$, $H$ the distribution function of $\Delta = \Delta(\cdot, F)$ and let $\hat{H}_n$ be the empirical distribution function of $\Delta_i$ ($1 \le i \le n$). We consider two finite sample version of the Gervini-Yohai depth, i.e., 
\begin{equation*}
d_{GY}(\vect{t}, \hat{F}_n, G) = 1 - G(\Delta(\vect{t}, \hat{F}_n)) ,
\end{equation*}
and
\begin{equation*}
d_{GY}(\vect{t}, \hat{F}_n, \hat{H}_n) = 1 - \hat{H}_n(\Delta(\vect{t}, \hat{F}_n)) .
\end{equation*}
The proportion of flagged $d$-variate outliers is defined by
\begin{equation*}
d_n = \sup_{\vect{t} \in A} \{ d_{GY}(\vect{t}, \hat{F}_n, \hat{H}_n) - d_{GY}(\vect{t}, \hat{F}_n, G) \}^+ .
\end{equation*}
Here $A = \{ \vect{t} \in \mathbb{R}^d: d_{GY}(\vect{t}, F, G) \leq d_{GY}(\vect{\zeta}, F, G) \}$, where $\vect{\zeta}$ is any point in $\mathbb{R}^d$ such that $\Delta(\vect{\zeta}, F)  = \eta$ and $\eta = G^{-1}(\alpha)$ is a large quantile of $G$. Then, we flag $\lfloor nd_{n} \rfloor$ observations. It is easy to see that,
\begin{align*}
d_n & = \sup_{\vect{t} \in A} \{ [1 - \hat{H}_n(\Delta(\vect{t}, \hat{F}_n))] - [1 - G(\Delta(\vect{t}, \hat{F}_n))] \}^+ \\
& =  \sup_{\vect{t} \in A} \{ G(\Delta(\vect{t}, \hat{F}_n)) - \hat{H}_n(\Delta(\vect{t}, \hat{F}_n)) \}^+ \\
& = \sup_{\Delta \ge \eta} \{ G(\Delta) - \hat{H}_n(\Delta) \}^+
\end{align*}
since $d_{GY}$ is a non increasing function of the squared Mahalanobis distance of the point $\vect{t}$.

\begin{remark}
  \label{rem:gn}
  In principle, $G_n$ could be any sequence of discrete distributions and for this reason we require that it satisfies condition (\ref{eqn:glivenko-cantelli}). If $G_n$ coincides with the empirical distribution of $G$, indicated as $\hat{G}_n$, such condition holds for the Glivenko-Cantelli lemma.
\end{remark}

\begin{remark}
  \label{rem:mahalanobis depth}
  The Mahalanobis depth is defined as \citep{Zuo2000a}
  \begin{equation*}
MHD(\vect{x},F) = \left(1 + \Delta(\vect{x},\vect{\mu}(F),\mat{\Sigma}(F)\right)^{-1},\ x \in \mathbb{R}^d.
  \end{equation*}
  Note that, for a continuous distribution $F$, MHD is equivalent to the GY-depth. But the Mahalanobis depth, which is completely parametric, cannot be used in our approach to define filters.
\end{remark}

We can rephrase Proposition 2 in \citet{Zamar2017}, that states the consistency property of the filter, as follows.
\begin{proposition}
Consider a random vector $\vect{Y} = (X_{1}, \ldots, X_{d}) \sim F_0$ and a pair of location and scatter estimators $\vect{T}_{0n}$ and $\mat{C}_{0n}$ such that $\vect{T}_{0n} \rightarrow \vect{\mu}_0 = \vect{\mu}(F_0) \in \mathbb{R}^d$ and $\mat{C}_{0n} \rightarrow \mat{\Sigma}_0 = \mat{\Sigma}(F_0)$ a.s.. Consider any continuous distribution function $G$ and let $\hat{H}_n$ be the empirical distribution function of $\Delta_i$ and $H_0(t) = \Pr ((\vect{Y} - \vect{\mu}_0)^t \mat{\Sigma}_0^{-1}(\vect{Y} - \vect{\mu}_0) \le t )$. If the distribution $G$ satisfies: 
\begin{equation}
\label{equ:prop2}
\max_{\vect{t} \in A} \{ d_{GY}(\vect{t},F_0,H_0) - d_{GY}(\vect{t},F_0,G) \} \le 0 ,
\end{equation}
where $A = \{ \vect{t} \in \mathbb{R}^d: d_{GY}(\vect{t}, F_0, G) \leq d_{GY}(\vect{\zeta}, F_0, G) \}$, where $\vect{\zeta}$ is any point in $\mathbb{R}^d$ such that $\Delta(\vect{\zeta}, F_0)  = \eta$ and $\eta = G^{-1}(\alpha)$ is a large quantile of $G$, then
\begin{equation*}
\frac{n_0}{n} \rightarrow 0 \qquad \text{a.s.}
\end{equation*}
where
\begin{equation*}
n_0 = \lfloor nd_{n} \rfloor .
\end{equation*}
\end{proposition} 

\begin{proof}.
Note that
\begin{equation*}
d_{GY}(\vect{t},\hat{F}_n,\hat{H}_n) - d_{GY}(\vect{t},\hat{F}_n, G) = G(\Delta(\vect{t},\vect{T}_{0n},\mat{C}_{0n})) - \hat{H}_n(\Delta(\vect{t},\vect{T}_{0n},\mat{C}_{0n})) 
\end{equation*}		
and condition in equation (\ref{equ:prop2}) is equivalent to 
\begin{equation*}
\max_{\Delta \ge \eta} \{G(\Delta) - H_0(\Delta) \} \le 0 ,
\end{equation*}
The rest of the proof is the same as in Proposition 2 of \citet{Zamar2017}.
\qed
\end{proof}

\section{Filters based on Half-space Depth}
\label{sec:halfspace}

In this section, we are going to give an example of depth-filter considering the half-space depth $d_{HS}(\cdot,F)$. In particular, we will prove the consistency property for this case.

\begin{definition}[Half-space depth]
Let $\vect{X}$ be a $\mathbb{R}^d$-valued random variable and let $F$ be a distribution function. For a point $\vect{x} \in \mathbb{R}^d$, the half-space depth of $\vect{x}$ with respect to $F$ is defined as the minimum probability of all closed half-spaces including $\vect{x}$:
\begin{equation*}
d_{HS} (\vect{x};F) = \min_{H \in \mathcal{H}(\vect{x})} P_F(\vect{X} \in H).
\end{equation*}
where $\mathcal{H}(\vect{x})$ indicates the set of all half-spaces in $\mathbb{R}^d$ containing $\vect{x} \in \mathbb{R}^d$.
\end{definition}

Given an independent and identically distributed sample $\vect{X}_1, \ldots, \vect{X}_n$, we define the filter in general dimension $d$ introduced previously, where here we use the half-space depth, as
\begin{equation}
  \label{eqn:dn-halfspace}
d_n = \sup_{\vect{x} \in C^\beta(F)} \{ d_{HS}(\vect{x}; \hat{F}_n) - d_{HS}(\vect{x}; F) \}^+ ,
\end{equation}
where $\beta$ is a high order probability, $\hat{F}_n(\cdot)$ is the empirical distribution function and $F$ is a chosen reference distribution which might depends, according to the assumed models, on unknown parameters, as in the case of location and dispersion models. In this last case, initial location and dispersion estimators, $\vect{T}_{0n}$ and $\mat{C}_{0n}$, are needed. As usual, $n_0 = \lfloor n d_n/2m \rfloor = \lfloor n d_n \rfloor$ observations with the smallest population depth are marked as outliers. Let $F_0$ be the true distribution of $X$, i.e. $\vect{X} \sim F_0$. Note that, so far we have no conditions on $F_0$. Here, we will prove the consistency property of the HS-filter when $\vect{X}$ is elliptically symmetric distributed.

\begin{definition}
  A random vector $\vect{X} \in \mathbb{R}^d$ is said elliptically symmetric distributed, denoted by $\vect{X} \sim E_d(h_0, \vect{\mu}, \mat{\Sigma})$, if it has a density function given by
  \begin{equation*}
    f_0(\vect{x}) \propto |\mat{\Sigma}^{-1/2}| h_0((\vect{x}-\vect{\mu})^\top\mat{\Sigma}^{-1}(\vect{x}-\vect{\mu})) .
  \end{equation*}
  where the density generating function $h_0$ is a non-negative scalar function, $\vect{\mu}$ is the location parameter vector and $\mat{\Sigma}$ is a $d \times d$ positive definite matrix.
\end{definition}

Let $\vect{X} \sim E_d(h_0,\vect{\mu},\mat{\Sigma})$. Denote by $F_0$ its distribution function and by $\Delta_{\vect{x}} = (\vect{x} - \vect{\mu})^\top \mat{\Sigma}^{-1} (\vect{x} - \vect{\mu})$ the squared Mahalanobis distance of a $d$-dimensional point $\vect{x}$. By Theorem 3.3 of \citet{Zuo2000b}, if a depth $d(\cdot,\cdot)$ is affine equivariant (P1) and has maximum at $\vect{\mu}$ (P2) (see Supplementary Material - Section SM--1) then the depth is of the form $d(\vect{x}; F_0) = g(\Delta_{\vect{x}})$ for some non increasing function $g$. In this case, we can restrict ourselves, without loss of generality, to the case $\vect{\mu} = \vect{0}$ and $\mat{\Sigma} = \mat{I}$, where $\mat{I}$ is the identity matrix of dimension $d$. Under this setting, it is easy to see that the half-space depth of a given point $\vect{x}$ is given by
\begin{equation}
  \label{eqn:hs-elliptically}
  d_{HS}(\vect{x}; F_0) = 1 - F_{0,1}(\sqrt{\Delta_{\vect{x}}}),
\end{equation}
where $F_{0,1}$ is a marginal distribution of $\vect{X}$. Denoting the reference distribution by $F$, let $f \propto h(\Delta_x) $ be the corresponding density function. Note that, if the function $h$ is such that
\begin{equation}
  \label{eqn:h-condition}
\frac{h_0(\Delta_{\vect{x}})}{h(\Delta_{\vect{x}})} \rightarrow 0  \qquad \Delta_{\vect{x}} \rightarrow \infty ,
\end{equation}
then, there exists a $\Delta^\ast$ such that, for all $\vect{x}$ with $\Delta_{\vect{x}} > \Delta^\ast$
\begin{equation*}
  d_{HS}(\vect{x}; F) \ge d_{HS}(\vect{x}; F_0).
\end{equation*}
Hence,
\begin{equation*}
\sup_{\{\vect{x}: \Delta_{\vect{x}} > \Delta^\ast\}} [d_{HS}(\vect{x}; F_0) - d_{HS}(\vect{x}; F)] \le 0
\end{equation*}
and therefore, for all $\beta > 1 - 2 F_{0,1}(-\sqrt{\Delta^{\ast}})$,  
\begin{equation*}
\sup_{C^\beta(F)} [d_{HS}(\vect{x}; F_0) - d_{HS}(\vect{x}; F)] \le 0 .
\end{equation*}

In order to compute the value $d_n$, we have to identify the set $C^\beta(F) = \{ \vect{x} \in \mathbb{R}^p: d_{HS}(\vect{x},F) \le d_{HS}(\eta_\beta,F) \}$ where $\eta_\beta$ is such that the probability of $C^\beta(F)$ is equal to $1-\beta$. In case we use the normal distribution as reference distribution, that is $F = N(\vect{T}_{0n}, \mat{C}_{0n})$, then, by Corollary 4.3 in \citet{Zuo2000b}, the computation of $C^\beta(F)$ is particularly simple. In fact, denoting with $\Delta_x = (\vect{x} - \vect{T}_{0n})^\top \mat{C}_{0n}^{-1}(\vect{x} - \vect{T}_{0n})$ the squared Mahalanobis distance of $\vect{x}$ using the initial location and dispersion estimates, the set $C^\beta(F)$ can be rewritten as
\begin{equation}
  \label{eqn:cbeta}
  C^\beta(F) = \{ \vect{x} \in \mathbb{R}^p: \Delta_x \ge (\chi^2_d)^{-1}(\beta) \},
\end{equation}
where $(\chi^2_d)^{-1}(\beta)$ is a large quantile of a chi-squared distribution with $d$ degrees of freedom. Now, we can state the consistency property for the HS-filter.

\begin{proposition}
  \label{prop:consistency-halfspace}
  Consider a random vector $(\vect{X}_1, \ldots , \vect{X}_n) \sim F_0(\vect{\mu}_0, \mat{\Sigma}_0)$ and suppose that $F_0$ is an elliptically symmetric distribution. Also consider a pair of location and dispersion estimators $\vect{T}_{0n}$ and $\mat{C}_{0n}$ such that $\vect{T}_{0n} \rightarrow \vect{\mu}_0$ and $\mat{C}_{0n} \rightarrow \mat{\Sigma}_0$ a.s.. Let $F$ be a chosen reference distribution and $\hat{F}_n$ the empirical distribution function. Assume that $F(\vect{\mu},\mat{\Sigma})$ is continuous with respect to $\vect{\mu}$ and $\mat{\Sigma}$. If the reference distribution satisfies
  \begin{equation}
    \label{eqn:prop-assumption}
\sup_{\vect{x} \in C^\beta(F)} [d_{HS}(\vect{x}; F_0) - d_{HS}(\vect{x}; F)] \le 0 
\end{equation}
where $\beta$ is some large probability, then   
\begin{equation*}
\frac{n_0}{n} \rightarrow 0 \mbox{ as } n \rightarrow \infty
\end{equation*}
where $n_0 = \lfloor n d_n \rfloor$.
\end{proposition}

\begin{proof}.
In \citet{Donoho1992}, it is proved that for $\vect{X}_1, \vect{X}_2, ... , \vect{X}_n$ i.i.d. with distribution $F_0$, as $n \rightarrow \infty$
\begin{equation*}
\sup_{\vect{t} \in \mathbb{R}^d} |d_{HS}(\vect{t},F_0) - d_{HS}(\vect{t},\hat{F}_n)| \rightarrow 0   \mbox{   a.s.}
\end{equation*}
Note that, by the continuity of $F$, $F(\vect{T}_{0n}, \mat{C}_{0n}) \rightarrow F(\vect{\mu}_0, \mat{\Sigma}_0)$ a.s..
Hence, for each $\varepsilon > 0$ there exists $n^\ast$ such that for all $n > n^\ast$ we have
\begin{align*}
  \sup_{\vect{x} \in C^\beta(F)} \{ d_{HS}(\vect{x}; & \hat{F}_n) - d_{HS}(\vect{x}; F(\vect{T}_{0n}, \mat{C}_{0n})) \} \le \\
  & \sup_{\vect{x} \in C^\beta(F)} \{ d_{HS}(\vect{x}; \hat{F}_n) - d_{HS}(\vect{x}; F_0(\vect{\mu}_0, \mat{\Sigma}_0)) \} + \\
															& \sup_{\vect{x} \in C^\beta(F)} \{ d_{HS}(\vect{x}; F_0(\vect{\mu}_0, \mat{\Sigma}_0)) - d_{HS}(\vect{x}; F(\vect{\mu}_0, \mat{\Sigma}_0)) \} + \\
															& \sup_{\vect{x} \in C^\beta(F)} \{ d_{HS}(\vect{x}; F(\vect{\mu}_0, \mat{\Sigma}_0)) - d_{HS}(\vect{x}; F(\vect{T}_{0n}, \mat{C}_{0n})) \}  \\
															\le    & \frac{\varepsilon}{2} + 0 + \frac{\varepsilon}{2} = \varepsilon
\end{align*}
which implies that $d_n = \sup_{\vect{x} \in C^\beta(F)} \{ d_{HS}(\vect{x}; \hat{F}_n) - d_{HS}(\vect{x}; F(\vect{T}_{0n}, \mat{C}_{0n})) \}^+$ goes to zero as $n \to \infty$. Hence, $\frac{n_0}{n} \to 0$ as $n \to \infty$.
\qed
\end{proof}

\begin{remark}
  \label{rem:normal-reference}
  We showed that if condition (\ref{eqn:h-condition}) holds, then assumption (\ref{eqn:prop-assumption}) of Proposition \ref{prop:consistency-halfspace} is satisfied. In other words, even if the actual distribution is unknown, asymptotically, the filter will not wrongly flag any outlier when the tail of the chosen reference distribution are heavier than that of the actual distribution. In case $F$ coincides to $F_0$, assumption \ref{eqn:prop-assumption} is clearly satisfied. We suggest to use for $F$ the same distribution assumed for the model of the data.
\end{remark}

\begin{remark}
  \label{rem:init-estimator}
When the underlying $F_0$ distribution is elliptical, a natural choice for $T_{0n}$ and $C_{0n}$ is as follows. For an univariate filter, $d = 1$, $T_{0n}$ and $C_{0n}$ might be for example the median and the MAD. In our study, when $d > 1$, as $\vect{T}_{0n}$ and $\mat{C}_{0n}$ we adopted the observation with maximum half-space depth, since the half-space depth corresponds to a generalization of the median in multivariate space, and the estimate given by a generalized S-estimator, respectively. Notice that, these initial estimates satisfy the almost sure convergence assumption, under the nominal model $F_0$.
\end{remark}

In Section SM--4 of the Supplementary Material we added an example which illustrates the filter based on half-space depth for $d=1$. In this case, it is possible to control independently the left and the right tail of the distribution and equation (\ref{eqn:dn-halfspace}) assumes a simpler form. However, in our implementation, we always use the general formulation that does not make this distinction.

On the other hand, the computation of the sample half-space depth is demanding for $d > 1$, even in low dimensions, since it is based on all possible one-dimensional projections. Here, we decided to use the random Tukey depth introduced by \citet{Cuesta2008}, a random approximation of the exact sample half-space depth, implemented in the \texttt{R} \citep{cran} package \texttt{ddalpha} \citep{Lange2012}. The reason is that approximate algorithms seem to be promising and, as pointed out in \citet{Cuesta2008}, they may outperform exact algorithms in terms of computational time. Note that, the random Tukey depth is able to handle also the case $d=50$, even if the computational time slightly increases. More information about exact algorithms can be found in \citep{Dyckerhoff2016}. These algorithms allow the exact computation of half-space depth for moderate dimensions and sample sizes.

\subsection{A consistent univariate, bivariate and $p$-variate filter}
\label{sec:pvariate1}

Consider a sample $\vect{X}_1, \ldots, \vect{X}_n$ where $\vect{X}_i \in \mathbb{R}^p, i = 1, \ldots, n$. In this subsection, we describe a filtering procedure which consists in applying the $d$-dimensional HS-filter given in equation (\ref{eqn:dn-halfspace}) three times in sequence, using $d=1$, $d=2$ and $d=p$.

We first apply the univariate filter to each variable separately. Let $\vect{X}^{(j)} = \{X_{1j},\ldots,X_{nj}\}$, $j=1,\ldots,p$, be a single variable. The univariate filter will flag $\lfloor n d_{nj} \rfloor$ observations as outliers, where $d_{nj}$ is as in equation (\ref{eqn:dn-halfspace}), and these values are replaced by  NA's values. Note that, the initial location and variance estimators used here are the median and the MAD of $\vect{X}^{(j)}$. Filtered data are indicated through an auxiliary matrix $\mat{U}$ of zeros and ones, with zero corresponding to a NA value.

Next, we identify the bivariate outliers by iterating the filter over all possible pairs of variables. Consider a pair of variables $\mat{X}^{(jk)} = \{ (X_{ij},X_{ik}), i = 1, \ldots, n \}$. The initial location and dispersion estimators are, respectively, the observation with maximum depth and the $2 \times 2$ covariance matrix estimate $S$ computed by the generalized S-estimator on non-filtered data \textbf{$\mat{X}^{(jk)}$}. For bivariate points with no flagged components by the univariate filter, we apply the bivariate filter. Given the pair of variables $\mat{X}^{(jk)}$, $1 \le j < k \le p$, we compute the value $d_{n}^{(jk)}$ given in equation (\ref{eqn:dn-halfspace}). In particular, to compute the sample depth $d_{HS}(\cdot,\hat{F}_n)$ we use the random Tukey depth, as mentioned before, through the function \texttt{depth.halfspace} implemented in the \texttt{R}  package \texttt{ddalpha} \citep{Lange2012}.

Then, $n_0^{(jk)}$ couples will be identified as bivariate outliers. But, at the end, we want to identify the cells $(i,j)$ which have to be flagged as cell-wise outliers. The procedure used for this purpose is described in \citet{Zamar2017} and reported here. Let
\begin{equation*}
J = \{ (i,j,k) : (X_{ij},X_{ik}) \mbox{ is flagged as bivariate outlier} \}
\end{equation*}
be the set of triplets which identifies the pairs of cells flagged by the bivariate filter where $i = 1, \ldots, n$ indicates the row. For each cell $(i,j)$ in the data, we count the number of flagged pairs in the $i$-th row in which the considered cell is involved:
\begin{equation*}
m_{ij} = \#\{ k : (i,j,k) \in J\}.
\end{equation*}
In absence of contamination, $m_{ij}$ follows approximately a binomial distribution $Bin(\sum_{k \not = j}\mat{U}_{jk},\delta)$ where $\delta$ represents the overall proportion of cell-wise outliers undetected by the univariate filter. Hence, we flag the cell $(i,j)$ if $m_{ij} > c_{ij}$, where $c_{ij}$ is the $0.99$-quantile of $Bin(\sum_{k \not = j}\mat{U}_{jk},0.1)$.

Finally, we perform the $p$-variate filter to the full data matrix. Detected observations (rows) are directly flagged as $p$-variate (case-wise) outliers. We denote the procedure based on univariate, bivariate and $p$-variate filters as HS-UBPF.

\subsection{A  sequencing filtering procedure}
\label{sec:pvariate2}

Suppose we would like to apply a sequence of $k$ filters with different dimension $1 \le d_1 < d_2 < \ldots < d_k \le p$. For each $d_i$, $i = 1, \ldots, k$, the filter updates the data matrix adding NA values to the $d_i$-tuples identified as $d_i$-variate outliers. In this way, each filter applies only those $d_i$-tuples that have not been flagged as outliers by the filters with lower dimension.

Initial values for each procedures rather than $d_1$ would be obtained by using the observation with the maximum half-space depth for location and the estimate given by GSE for the scatter matrix.

This procedure aims to be a valid alternative to that used in the presented HS-UBPF filter to perform a sequence of filters with different dimensions. However, this is a preliminary idea, indeed it has not been implemented yet.

\section{Monte Carlo results}
\label{sec:simulation}

We performed a Monte Carlo simulation to assess the performance of the proposed filter based on half-space depth. After the filter flags the outlying observations, the generalized S-estimator is applied to the data with added missing values. Part of our simulation study is based on the same setup described in \citet{Zamar2017} since it seems a good choice to test our filter in presence of contamination and the comparison with previous methods is easier. In particular, we compare the filter introduced in \citet{Agostinelli2015a} (indicated as GY-UF in case of univariate filter and GY-UBF for univariate and bivariate filter) and the same filter with the improvements proposed in \citet{Zamar2017} (indicated here as GY-UBF-DDC-C) to the presented filter based on statistical data depth functions obtained using the half-space depth (HS-UF for the univariate filter, HS-UBF for the univariate-bivariate filter, HS-UBPF for the univariate-bivariate-$p$-variate filter and HS-UBPF-DDC-C for the combination of the HS-UBPF with the modifications in \citet{Zamar2017}). The already existing filters are implemented in the \texttt{R} \citep{cran} package \texttt{GSE} \citep{Leung2015}, whereas the \texttt{R} code for the proposed filter based on half-space depth is available in the \texttt{R} package \texttt{GSEdepth} provided as supplementary material.

We considered samples from a $N_p(\vect{0}, \mat{\Sigma}_0)$, where all values in  $diag(\mat{\Sigma}_0)$ are equal to $1$, $p = 10, 20, 30, 40, 50$ and the sample size is $n = 10p$. Since our model is the normal distribution, we choose the normal distribution as reference distribution. We consider the following scenarios:
\begin{itemize}
\item Clean data: data without changes.
\item Cell-Wise contamination: a proportion $\epsilon$ of cells in the data is replaced by $X_{ij} \sim N(k,0.1^2)$, where $k = 1, \ldots, 10$.
\item Case-Wise contamination: a proportion $\epsilon$ of cases in the data is replaced by $\vect{X}_i \sim 0.5N(c\vect{v},0.1^2\mat{I})\ +\ 0.5N(-c\vect{v},0.1^2\mat{I})$, where $c = \sqrt{k(\chi^2_p)^{-1}(0.99)}$, $k = 2, 4, \ldots,100$ and $\vect{v}$ is the eigenvector corresponding to the smallest eigenvalue of $\mat{\Sigma}_0$ with length such that $(\vect{v}-\vect{\mu}_0)^\top\mat{\Sigma}_0^{-1}(\vect{v}-\vect{\mu}_0) = 1$.
  \item Mixed contamination: case-wise and cell-wise contamination are introduced at the same time (after replacing a proportion of cases, a proportion of the remaining cells is contaminated).
\end{itemize}

The proportions of contaminated rows chosen for case-wise contamination are $\epsilon = 0.1, 0.2$, and $\epsilon = 0.02,0.05, 0.1$ for cell-wise contamination. For the mixed contamination, we combined the proportions $\epsilon = 0.05, 0.1$ and $\epsilon = 0.02, 0.05$ for case-wise and cell-wise contamination, respectively. Finally, we tested the behaviour of the procedure for increasing $n$. We considered $p=5$ variables and $n = (10p, 50p, 100p)$ observations. Case-wise contamination and cell-wise contamination scenarios, as explained above, were performed on this setting. The number of replicates in our simulation study is $N=200$.

We measure the performance of a given pair of location and scatter estimators $\hat{\vect{\mu}}$ and $\hat{\mat{\Sigma}}$ using the mean squared error (MSE) and the likelihood ratio test (LRT) distance:
\begin{align*}
& MSE = \frac{1}{N}\sum_{i=1}^N (\hat{\vect{\mu}}_i - \vect{\mu}_0)^\top (\hat{\vect{\mu}}_i - \vect{\mu}_0) \\  
& LRT(\hat{\mat{\Sigma}},\mat{\Sigma}_0) = \frac{1}{N}\sum_{i=1}^N D(\hat{\mat{\Sigma}}_i,\mat{\Sigma}_0)
\end{align*}
where $\hat{\mat{\Sigma}}_i$ is the estimate of the $i$-th replication and $D(\mat{\Sigma},\mat{\Sigma}_0)$ is the Kullback-Leibler divergence between two Gaussian distributions with the same mean and variances $\mat{\Sigma}$ and $\mat{\Sigma}_0$. Finally, we computed the maximum average LRT distances and maximum average MSE considering all contamination values $k$.

\begin{table}
\begin{adjustbox}{width=\columnwidth,center}    
\begin{tabular}{rr|rr|rr|r|rr|r|}
  \hline
  &  & \multicolumn{2}{c|}{UF} & \multicolumn{2}{c|}{UBF} &  & \multicolumn{2}{c|}{DDC-C} & \\
$p$ & $\epsilon$ & GY & HS & GY & HS & HS-UBPF & GY-UBF & HS-UBPF & MLE \\
  \hline
10 & 0 & 0.8 & 0.7 & 0.9 & 0.7 & 0.8 & 1.0 & 1.0 & 0.6 \\ 
   & 0.02 & 1.2 & 1.1 & 1.3 & 1.1 & 1.1 & 1.1 & 1.1 & 113.0 \\ 
   & 0.05 & 4.6 & 4.8 & 4.6 & 4.9 & 4.8 & 2.4 & 2.5 & 290.5 \\ 
   & 0.1 & 16.4 & 16.7 & 16.4 & 16.9 & 16.8 & 13.3 & 13.2 & 555.3 \\ 
20 & 0 & 1.3 & 1.2 & 1.4 & 1.3 & 1.3 & 1.8 & 1.8 & 1.1 \\ 
   & 0.02 & 3.9 & 3.8 & 4.2 & 4.0 & 3.8 & 2.5 & 2.5 & 146.4 \\ 
   & 0.05 & 11.0 & 11.3 & 11.3 & 11.6 & 11.4 & 8.2 & 8.3 & 380.8 \\ 
   & 0.1 & 24.4 & 24.6 & 24.5 & 25.1 & 24.7 & 21.6 & 21.8 & 742.7 \\ 
30 & 0 & 1.9 & 1.8 & 2.0 & 1.9 & 1.9 & 3.4 & 3.4 & 1.6 \\ 
   & 0.02 & 6.0 & 5.8 & 6.5 & 6.1 & 5.8 & 5.0 & 5.1 & 179.5 \\ 
   & 0.05 & 14.5 & 14.7 & 15.1 & 15.3 & 14.9 & 13.4 & 13.4 & 470.5 \\ 
   & 0.1 & 30.5 & 30.6 & 30.5 & 31.4 & 31.0 & 31.1 & 31.5 & 930.4 \\ 
40 & 0 & 2.4 & 2.3 & 2.6 & 2.4 & 2.5 & 5.8 & 5.8 & 2.1 \\ 
   & 0.02 & 7.5 & 7.4 & 8.2 & 7.8 & 7.4 & 9.2 & 9.2 & 213.2 \\ 
   & 0.05 & 17.4 & 17.7 & 18.1 & 18.3 & 17.9 & 20.0 & 20.1 & 565.0 \\ 
   & 0.1 & 35.6 & 35.7 & 35.6 & 36.5 & 36.1 & 41.4 & 42.4 & 1117.5 \\ 
50 & 0 & 2.9 & 2.8 & 3.1 & 3.0 & 3.0 & 5.1 & 5.0 & 2.6 \\ 
   & 0.02 & 8.8 & 8.6 & 9.7 & 9.1 & 8.8 & 12.2 & 12.3 & 245.7 \\ 
   & 0.05 & 19.9 & 20.1 & 20.8 & 21.0 & 20.7 & 24.5 & 24.5 & 653.0 \\ 
   & 0.1 & 40.0 & 40.1 & 40.0 & 41.0 & 40.6 & 44.7 & 44.3 & 1291.1 \\ 
   \hline
\end{tabular}
\end{adjustbox}
\caption{Maximum average LRT distance under cell-wise contamination.} 
\label{tab:max_lrt_cell}
\end{table}

\begin{figure}
\begin{center}
\includegraphics[width=0.49\textwidth]{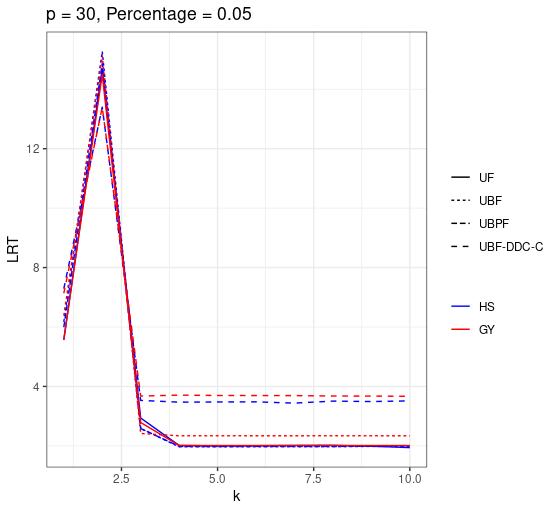}
\includegraphics[width=0.5\textwidth]{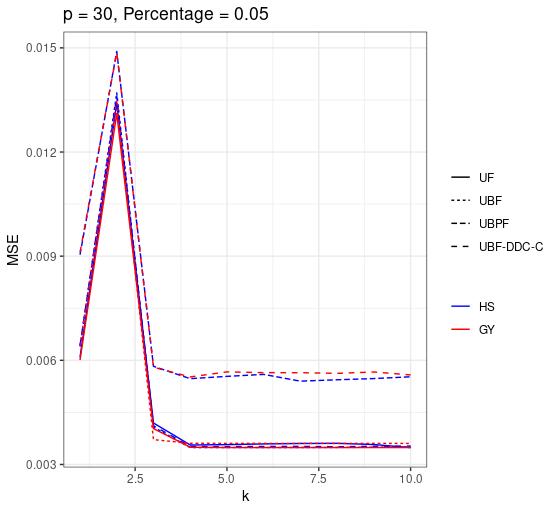}
\end{center}
\caption{Average LRT (left) and average MSE (right) versus the contamination value $k$, for $5\%$ cell-wise contamination level and $p=30$}
\label{fig:k-30-icm-0.05}
\end{figure}

Table \ref{tab:max_lrt_cell} shows the maximum average LRT distances under cell-wise contamination.  The univariate and univariate-bivariate filters have a similar behaviour, while HS-UBPF has a lightly better performance. GY-UBF-DDC-C and HS-UBPF-DDC-C have lower maximum average LRT distances if the number of variables is not large, but their LRT distances are higher with respect to the other filters for large $k$. This behavior is shown in Figure \ref{fig:k-30-icm-0.05} (left) where the average LRT distances versus different contamination values are displayed, with $5\%$ of cell-wise contamination level and $p=30$.

\begin{table}
\begin{adjustbox}{width=\columnwidth,center}
\begin{tabular}{rr|rr|rr|r|rr|r|}
  \hline
  &  & \multicolumn{2}{c|}{UF} & \multicolumn{2}{c|}{UBF} &  & \multicolumn{2}{c|}{DDC-C}& \\
$p$ & $\epsilon$ & GY & HS & GY & HS & HS-UBPF & GY-UBF & HS-UBPF & MLE \\ 
  \hline
10 & 0 & 0.8 & 0.7 & 0.9 & 0.7 & 0.8 & 1.0 & 1.0 & 0.6 \\ 
   & 0.1 & 9.8 & 7.6 & 14.9 & 8.5 & 6.2 & 3.5 & 3.4 & 893.9 \\ 
   & 0.2 & 93.0 & 79.6 & 161.1 & 120.1 & 77.1 & 18.7 & 17.5 & 1593.6 \\ 
20 & 0 & 1.3 & 1.2 & 1.4 & 1.3 & 1.3 & 1.8 & 1.8 & 1.1 \\ 
   & 0.1 & 25.7 & 21.2 & 38.1 & 27.2 & 26.0 & 6.8 & 6.9 & 894.1 \\ 
   & 0.2 & 368.0 & 322.3 & 428.9 & 441.0 & 373.8 & 19.6 & 20.1 & 1593.8 \\ 
30 & 0 & 1.9 & 1.8 & 2.0 & 1.9 & 1.9 & 3.4 & 3.4 & 1.6 \\ 
   & 0.1 & 50.8 & 44.9 & 64.0 & 70.3 & 68.6 & 9.0 & 8.7 & 895.0 \\ 
   & 0.2 & 745.8 & 708.7 & 620.0 & 744.2 & 751.3 & 17.1 & 17.6 & 1595.1 \\ 
40 & 0 & 2.4 & 2.3 & 2.6 & 2.4 & 2.5 & 5.8 & 5.8 & 2.1 \\ 
   & 0.1 & 64.2 & 89.8 & 97.0 & 70.7 & 67.7 & 16.2 & 16.3 & 898.0 \\ 
   & 0.2 & 1156.9 & 1112.1 & 852.0 & 1078.4 & 1088.0 & 22.7 & 21.4 & 1600.2 \\ 
50 & 0 & 2.9 & 2.8 & 3.1 & 3.0 & 3.0 & 5.1 & 4.8 & 2.6 \\ 
   & 0.1 & 175.2 & 215.6 & 123.3 & 156.6 & 163.9 & 30.5 & 29.9 & 898.0 \\ 
   & 0.2 & 1528.8 & 1468.0 & 1081.6 & 1354.5 & 1364.5 & 21.2 & 20.1 & 1599.9 \\ 
   \hline
\end{tabular}
\end{adjustbox}
\caption{Maximum average LRT distance under case-wise contamination.} 
\label{tab:max_lrt_case}
\end{table}

\begin{figure}
\begin{center}
\includegraphics[width=0.8\textwidth]{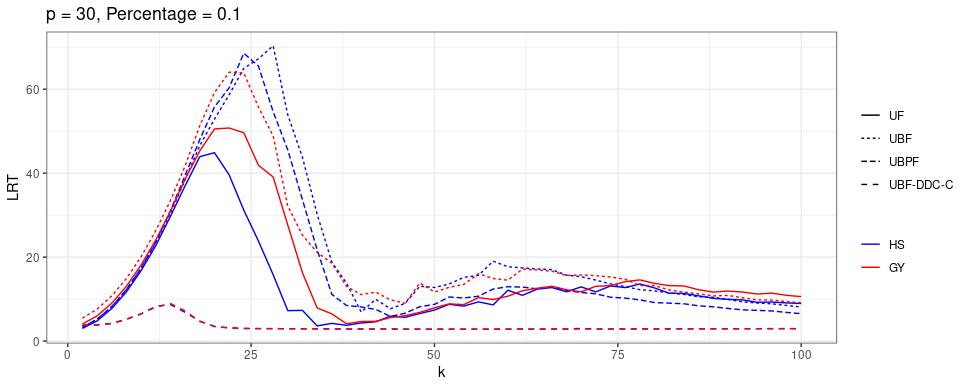}
\includegraphics[width=0.8\textwidth]{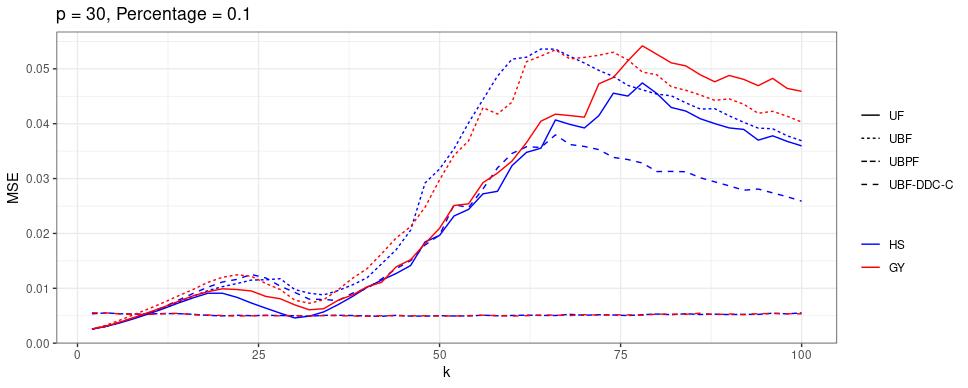}
\end{center}
\caption{Average LRT (top) and average MSE (bottom) versus the contamination value $k$, for $10\%$ case-wise contamination level and $p=30$}
\label{fig:k-30-thcm-0.1}
\end{figure}

Table \ref{tab:max_lrt_case} shows the maximum average LRT distances under case-wise contamination. Overall, the GY-UBP-DDC-C and HS-UBPF-DDC-C outperform all the other filters obtaining better results. An illustration of their behavior is given in Figure \ref{fig:k-30-thcm-0.1} (top) which shows the average LRT distances for different values of $k$, with $10\%$ of case-wise contamination level and $p=30$.

\begin{table}
\begin{adjustbox}{width=\columnwidth,center}
\begin{tabular}{rr|rr|rr|r|rr|r|}
  \hline
  &  & \multicolumn{2}{c|}{UF} & \multicolumn{2}{c|}{UBF} &  & \multicolumn{2}{c|}{DDC-C} & \\
$p$ & $\epsilon$ & GY & HS & GY & HS & HS-UBPF & GY-UBF & HS-UBPF & MLE\\
  \hline
10 & 0 & 1.1 & 1.1 & 1.1 & 1.1 & 1.1 & 1.3 & 1.3 & 1.0 \\ 
   & 0.02 & 1.3 & 1.3 & 1.3 & 1.3 & 1.3 & 1.5 & 1.5 & 6.8 \\ 
   & 0.05 & 1.9 & 2.0 & 2.0 & 2.0 & 2.0 & 2.0 & 2.0 & 30.2 \\ 
   & 0.1 & 4.8 & 4.9 & 4.8 & 4.9 & 4.9 & 5.0 & 5.0 & 109.2 \\ 
20 & 0 & 0.5 & 0.5 & 0.5 & 0.5 & 0.5 & 0.7 & 0.7 & 0.5 \\ 
   & 0.02 & 0.7 & 0.7 & 0.7 & 0.7 & 0.7 & 0.8 & 0.8 & 5.4 \\ 
   & 0.05 & 1.5 & 1.5 & 1.5 & 1.5 & 1.5 & 1.6 & 1.6 & 27.8 \\ 
   & 0.1 & 4.4 & 4.5 & 4.5 & 4.6 & 4.6 & 4.6 & 4.7 & 104.7 \\ 
30 & 0 & 0.3 & 0.3 & 0.4 & 0.3 & 0.4 & 0.6 & 0.6 & 0.3 \\ 
   & 0.02 & 0.5 & 0.5 & 0.5 & 0.5 & 0.5 & 0.7 & 0.7 & 4.9 \\ 
   & 0.05 & 1.3 & 1.3 & 1.3 & 1.4 & 1.4 & 1.5 & 1.5 & 26.8 \\ 
   & 0.1 & 4.3 & 4.3 & 4.3 & 4.4 & 4.4 & 4.5 & 4.7 & 103.2 \\ 
40 & 0  & 0.3 & 0.3 & 0.3 & 0.3 & 0.3 & 0.6 & 0.6 & 0.2 \\ 
   & 0.02 & 0.4 & 0.4 & 0.5 & 0.4 & 0.4 & 0.7 & 0.7 & 4.7 \\ 
   & 0.05 & 1.3 & 1.3 & 1.3 & 1.3 & 1.3 & 1.5 & 1.6 & 26.4 \\ 
   & 0.1 & 4.3 & 4.3 & 4.3 & 4.4 & 4.4 & 4.5 & 4.6 & 102.5 \\ 
50 & 0 & 0.2 & 0.2 & 0.2 & 0.2 & 0.2 & 0.4 & 0.3 & 0.2 \\ 
   & 0.02 & 0.4 & 0.4 & 0.4 & 0.4 & 0.4 & 0.6 & 0.6 & 4.6 \\ 
   & 0.05 & 1.2 & 1.2 & 1.2 & 1.3 & 1.3 & 1.4 & 1.4 & 26.1 \\ 
   & 0.1 & 4.2 & 4.2 & 4.2 & 4.4 & 4.3 & 4.3 & 4.5 & 101.9 \\
  \hline
\end{tabular}
\end{adjustbox}
\caption{Maximum average MSE distance under cell-wise contamination.} 
\label{tab:max_mse_cell}
\end{table}

\begin{table}
\begin{adjustbox}{width=\columnwidth,center}
\begin{tabular}{rr|rr|rr|r|rr|r|}
  \hline
  &  & \multicolumn{2}{c|}{UF} & \multicolumn{2}{c|}{UBF} &  & \multicolumn{2}{c|}{DDC-C} & \\
$p$ & $\epsilon$ & GY & HS & GY & HS & HS-UBPF & GY-UBF & HS-UBPF & MLE \\
\hline
10 & 0 & 1.1 & 1.1 & 1.1 & 1.1 & 1.1 & 1.3 & 1.3 & 1.0 \\ 
   & 0.1 & 2.8 & 2.5 & 3.2 & 2.9 & 1.9 & 1.9 & 1.9 & 21.8 \\ 
   & 0.2 & 15.1 & 14.2 & 20.1 & 16.1 & 9.7 & 2.5 & 2.8 & 84.4 \\ 
20 & 0 & 0.5 & 0.5 & 0.5 & 0.5 & 0.5 & 0.7 & 0.7 & 0.5 \\ 
   & 0.1 & 3.5 & 2.9 & 4.2 & 4.0 & 2.7 & 0.8 & 0.8 & 10.8 \\ 
   & 0.2 & 28.6 & 25.8 & 34.1 & 25.9 & 21.3 & 1.3 & 1.2 & 41.9 \\ 
30 & 0 & 0.3 & 0.3 & 0.4 & 0.3 & 0.4 & 0.6 & 0.6 & 0.3 \\ 
   & 0.1 & 5.4 & 4.7 & 5.3 & 5.4 & 3.8 & 0.6 & 0.6 & 7.1 \\ 
   & 0.2 & 50.6 & 46.7 & 37.2 & 46.5 & 48.0 & 0.8 & 0.8 & 27.6 \\ 
40 & 0 & 0.3 & 0.3 & 0.3 & 0.3 & 0.3 & 0.6 & 0.6 & 0.2 \\ 
   & 0.1 & 7.1 & 6.6 & 6.1 & 6.4 & 4.7 & 0.5 & 0.5 & 5.3 \\ 
   & 0.2 & 41.6 & 38.1 & 34.7 & 38.9 & 39.7 & 0.7 & 0.7 & 20.6 \\ 
50 & 0 & 0.2 & 0.2 & 0.2 & 0.2 & 0.2 & 0.4 & 0.3 & 0.2 \\ 
   & 0.1 & 7.9 & 7.6 & 6.3 & 6.2 & 5.0 & 0.5 & 0.5 & 4.3 \\ 
   & 0.2 & 32.4 & 30.0 & 30.6 & 31.9 & 32.5 & 0.5 & 0.5 & 16.5 \\ 
\hline
\end{tabular}
\end{adjustbox}
\caption{Maximum average MSE distance under case-wise contamination.} 
\label{tab:max_mse_case}
\end{table}

\begin{figure}
\begin{center}
\includegraphics[width=0.9\textwidth]{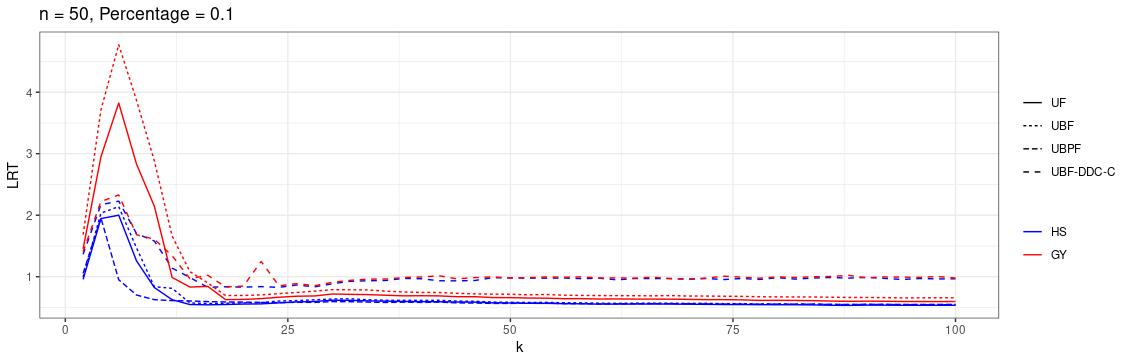}
\includegraphics[width=0.9\textwidth]{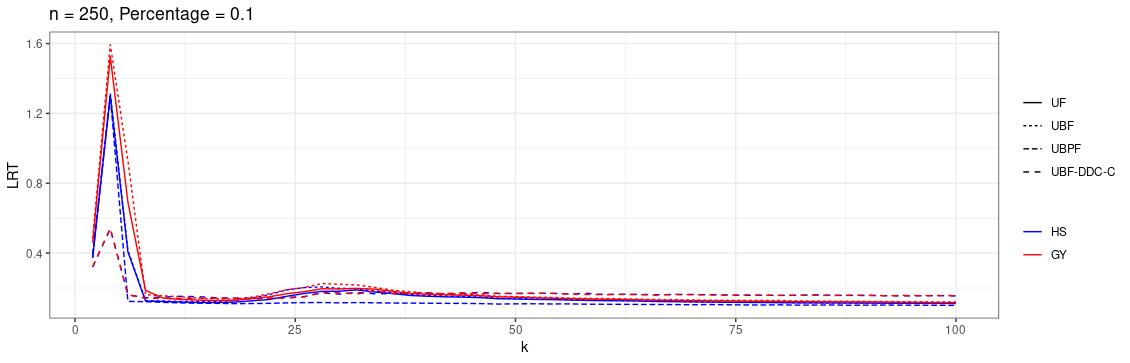}
\includegraphics[width=0.9\textwidth]{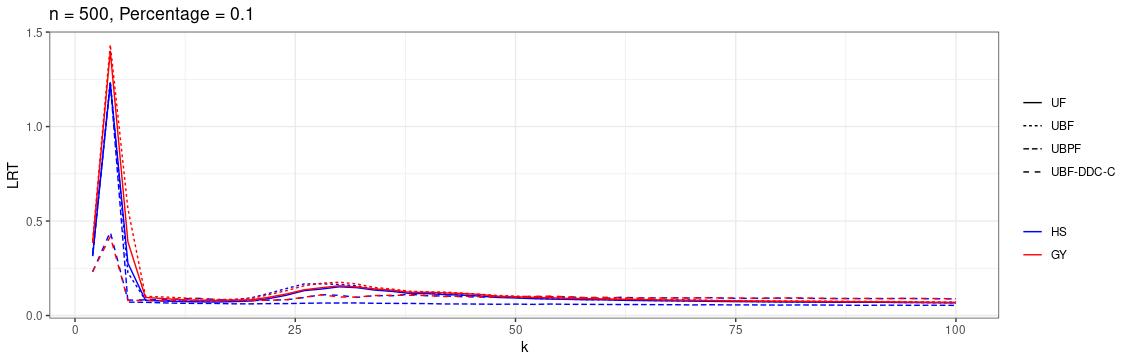}
\includegraphics[width=0.9\textwidth]{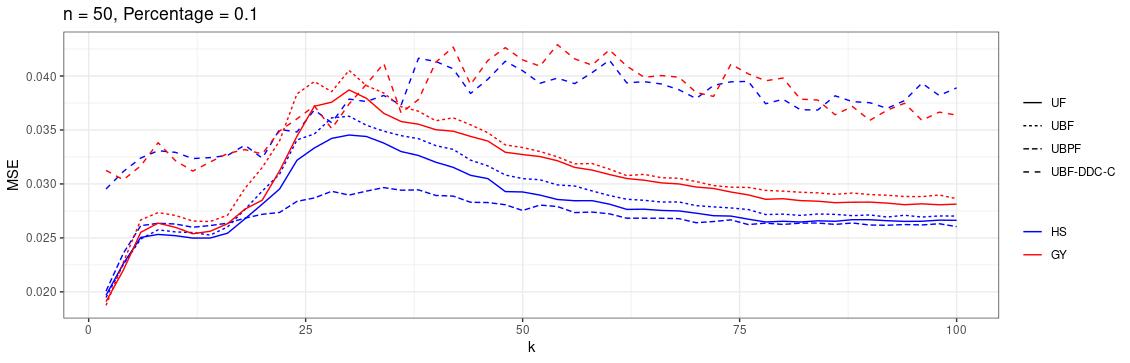}
\includegraphics[width=0.9\textwidth]{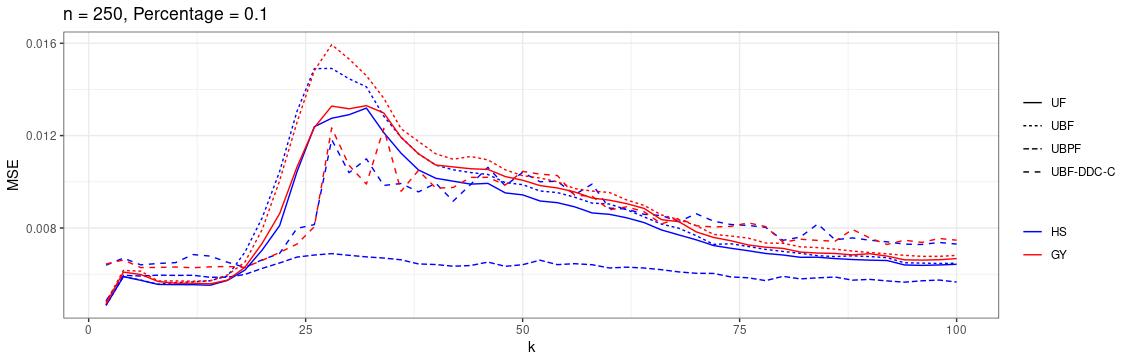}
\includegraphics[width=0.9\textwidth]{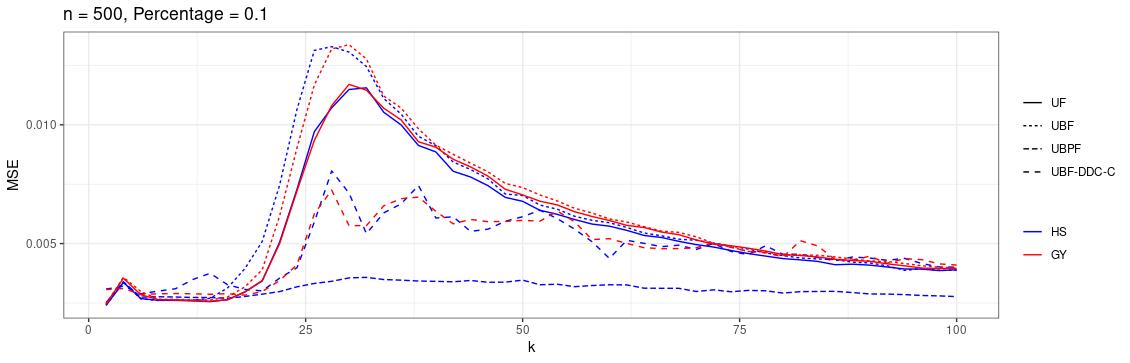}
\end{center}
\caption{Average LRT (top) and average MSE (bottom) in $0.1$ case-wise contamination level versus the contamination value $k$, for $p=5$ and $n = 50,250,500$.}
\label{fig:k-5-n-thcm-0.1}
\end{figure}

\begin{figure}
\begin{center}
\includegraphics[width=0.49\textwidth]{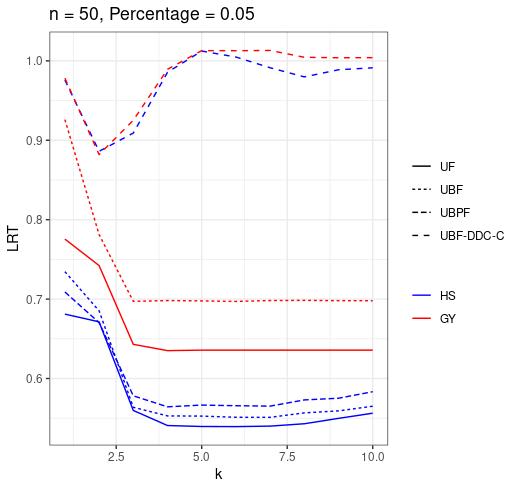}
\includegraphics[width=0.5\textwidth]{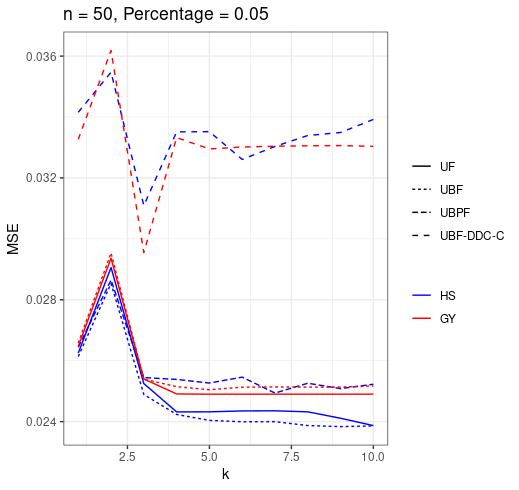}
\includegraphics[width=0.49\textwidth]{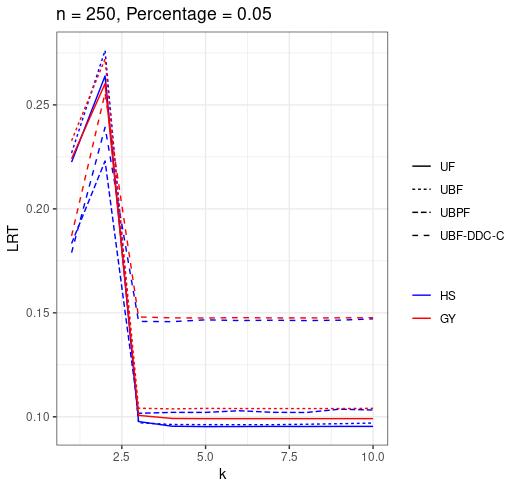}
\includegraphics[width=0.5\textwidth]{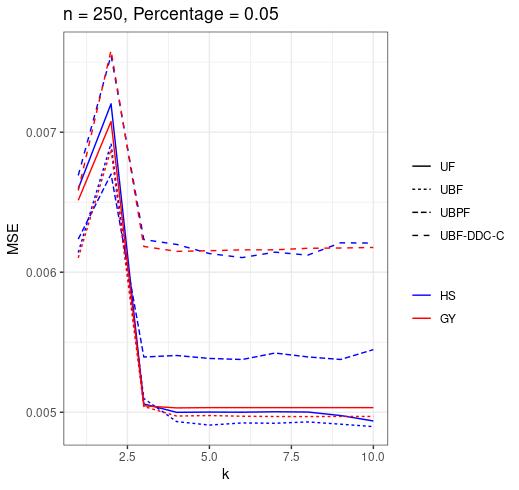}
\includegraphics[width=0.49\textwidth]{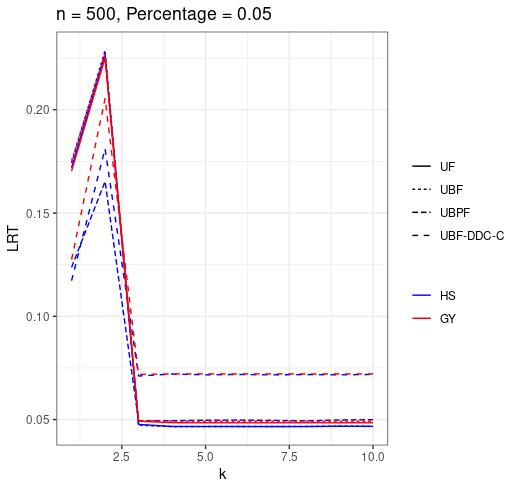}
\includegraphics[width=0.5\textwidth]{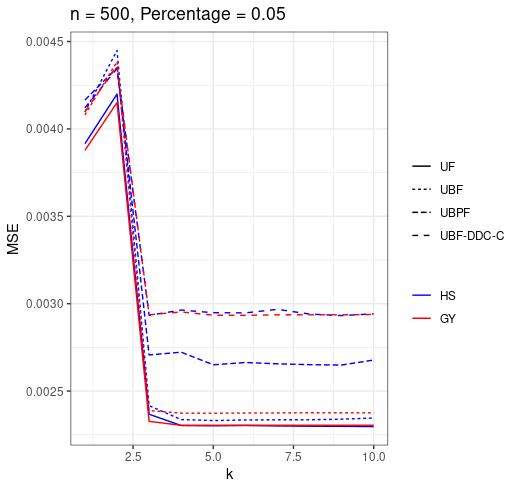}
\end{center}
\caption{Average LRT (left) and average MSE (right) in $0.05$ cell-wise contamination level versus the contamination value $k$, for $p=5$ and $n = 50,250,500$.}
\label{fig:k-5-n-icm-0.05}
\end{figure}

Table \ref{tab:max_mse_cell} and Table \ref{tab:max_mse_case} show the maximum average MSE under cell-wise and case-wise contamination, respectively. The values in the tables are the MSE values multiplied by 100 for a better visualization and model comparison. Under case-wise contamination, the GY-UBF-DDC-C and HS-UBPF-DDC-C outperform the other filters, and have also competitive results for cell-wise contamination. In Figure \ref{fig:k-30-icm-0.05} (right) and Figure \ref{fig:k-30-thcm-0.1} (bottom), the average MSE versus different contamination values $k$ are displayed, with $p=30$ and $0.05$ of cell-wise contamination and $0.1$ of case-wise contamination, respectively.

The results given by the mixed contamination scenario do not show any additional information and they are not reported.

Finally, Figure \ref{fig:k-5-n-thcm-0.1} and Figure \ref{fig:k-5-n-icm-0.05} show the average LRT and average MSE with respect to different value of $k$, for $10\%$ of case-wise contamination and $5\%$ of cell-wise contamination, respectively, for $p=5$ and different number of observations $n$. For increasing $n$, the filters perform better showing smaller average LRT and average MSE values. In particular, depth-filters present better improvements in case of case-wise contamination and they seem to perform better then those in combination with DDC.  

In a second Monte Carlo experiment, we use the location-scale family of multivariate Student's $t$-distribution with $5$ degrees of freedom as reference distribution $F$. We consider two data generation processes: in the first case data are simulated from the multivariate Normal distribution and in the second case data are simulated from a $t_5$ distribution with $5$ degrees of freedom. Apart from this, the setup of the experiment is the same of the previous one. The construction of the half-space-filter for this case follows directly from the definition given in equation (\ref{eqn:dn-halfspace}), with just one change. In particular, since the $t$ distribution belongs to the family of elliptically symmetric distribution, equation (\ref{eqn:hs-elliptically}) holds and it is used to compute the theoretical depth. On the other hand, the sample depth is again computed using the random Tukey depth. Complete results are not reported. In this new setup, the HS-filters are still competitive for casewise contamination, while they outperform the GY-filters in case of cellwise contamination. This performance does not change if observations are sampled from a Normal distribution or a $t$-distribution.

\section{Examples}
\label{sec:example}

In Subsection \ref{subsec:example-sn}, we illustrate how depth-filters approach can be used in models different from the location and scatter model with elliptical contours. In particular, we provide details of applying such filters to multivariate Skew-Normal distributions. A real-data application is reported  in Subsection \ref{subsec:example-small-cap}. The \texttt{R} package \texttt{GSEdepth}, available as supplementary material, implements the new procedures and contains the used data set. 

\subsection{\textbf{Multivariate Skew-Normal distributions}}
\label{subsec:example-sn}


In this example we consider a $p$-multivariate Skew-Normal random variable $\vect{X} \sim SN_p(\vect{\xi},\mat{\Omega},\vect{\alpha})$, with a location parameter $\vect{\xi}$, a positive definite scatter matrix $\mat{\Omega}$, and a skewness vector parameter $\vect{\alpha}$. We point out the reader to \citet{Azzalini2014} for the details on multivariate Skew-Normal distributions. The mean vector $\vect{\mu}$ and the covariance matrix $\mat{\Sigma}$ do not coincide with the distribution parameters, however they are easily evaluated as \citep[][formulas 2.27, 5.31 and 5.32]{Azzalini2014}
\begin{equation*}
  \vect{\mu} = \mathbb{E}(\vect{X}) = \vect{\xi} + \omega \vect{\nu}, \quad \Sigma = \mathbb{C}\mathrm{ov}(\vect{X}) = \mat{\Omega} - \omega \nu \nu^\top \omega \ ,
\end{equation*}
where $\nu = \sqrt{\frac{2}{\pi}} (1 + \vect{\alpha}^\top \bar{\Omega} \vect{\alpha})^{-1/2} \bar{\Omega} \vect{\alpha}$ while $\bar{\Omega}$ and $\omega$ are, respectively, the correlation matrix obtained from $\Omega$ and a diagonal matrix with the square-root of the diagonal elements of $\Omega$.  
We are going to apply the GY-filter and the HS-filter in this framework, using as reference distribution the skew-normal model, evaluated at the true parameters value. Subsection SM--6-1 of the Supplementary Material provides all the necessary code to replicate the results and the figures.

A sample of size $n = 100$ is obtained and it is represented in Figure \ref{fig:sn-contour} (blue crosses) together with the density contours (black dotted lines) and the half-space depth contours (red dashed lines).
\begin{figure}
\begin{center}
\includegraphics[width=0.7\textwidth]{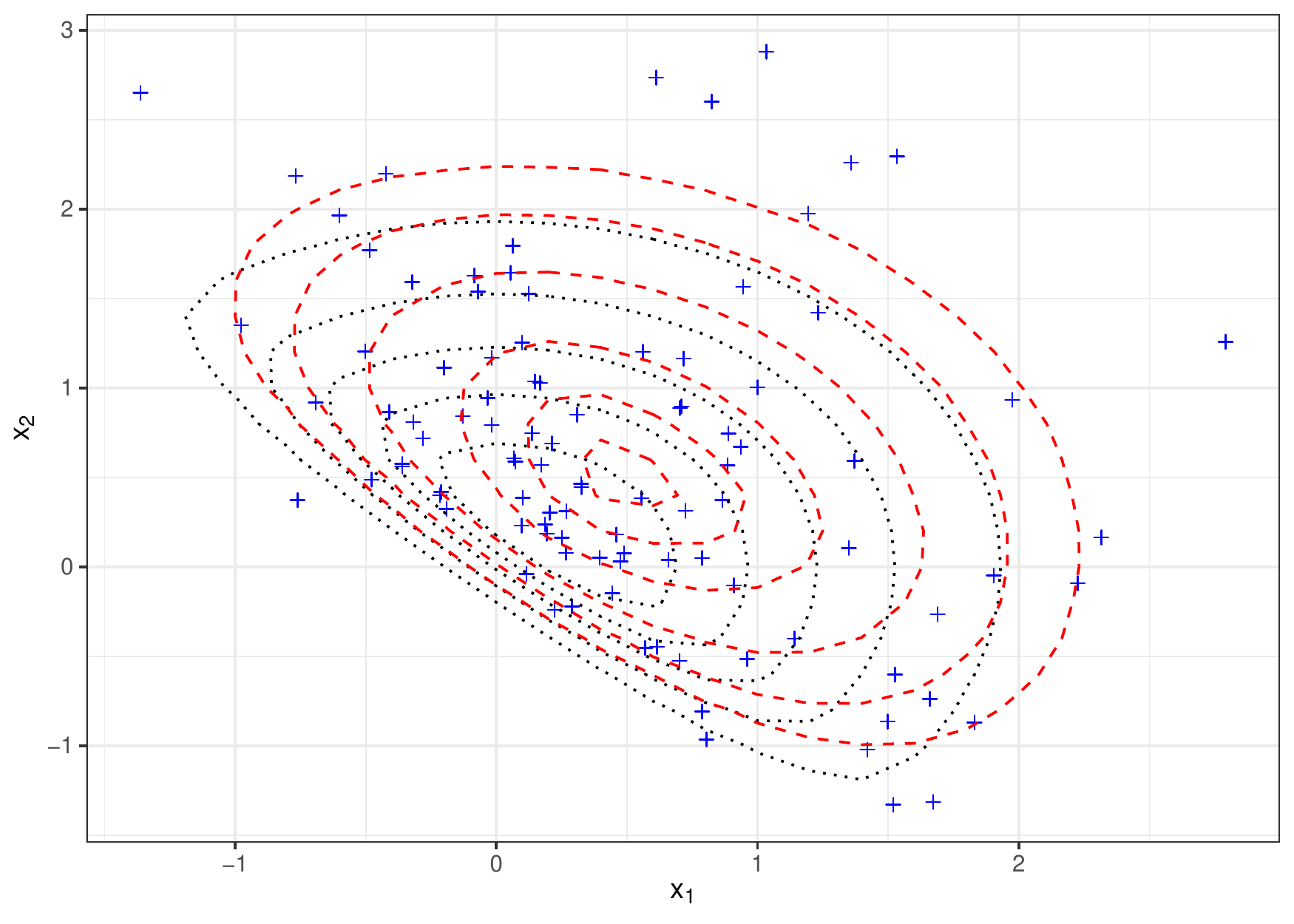}
\end{center}
\caption{Contour plot of the density of the skew-normal (black dotted lines) and of the half-space depth (red dashed lines). Sample observations are blue crosses.} 
\label{fig:sn-contour}
\end{figure}

The GY-filters, that are based on Mahalanobis distances, need the mean vector and the variance-covariance matrix to be computed. The half-space-depth filters work directly with the actual parametrization of the reference distribution. While the set $C^\beta(F)$ is always an ellipse for GY-filters, this is not the case for half-space-depth filters, which, instead, depends on the shape of the reference distribution, and in this case it is able to take into account the asymmetry of the Skew-Normal distribution.

We are going to add artificially $20$ outlying observations sampled from a $N_2((-0.2,-0.25), 0.01 \mat{I}_2)$ in an iterative procedure. Note that, these points, with high probability, lie inside the boundary set given by the Mahalanobis distance but outside the boundary set computed using the half-space depth. This position is clearly crucial, however it is a region of low density according to the true model. In each iteration an outlier is added to the data set and the number of flagged observations $n_0$ is computed and reported in Table \ref{tab:n0-sn}. The GY-filter is insensitive to this kind of outliers, indeed, the number of detected cells is stable or decreases as the number of added outliers increases. Vice versa, the number of detected cells by the HS-filter is almost always equal to the amount of added outliers.

\begin{table}[ht]
\centering
\begin{tabular}{rrrrrrrrrrr}
  \hline
$n^o$ of outliers & 1 & 2 & 3 & 4 & 5 & 6 & 7 & 8 & 9 & 10  \\ 
  \hline
GY-filter & 4 & 4 & 4 & 4 & 4 & 4 & 4 & 4 & 4 & 4  \\ 
  HS-filter & 3 & 3 & 3 & 3 & 4 & 5 & 6 & 7 & 8 & 9  \\ 
   \hline
$n^o$ of outliers & 11 & 12 & 13 & 14 & 15 & 16 & 17 & 18 & 19 & 20 \\ 
  \hline
GY-filter & 4 & 4 & 4 & 3 & 3 & 3 & 3 & 3 & 3 & 3 \\ 
  HS-filter & 10 & 11 & 12 & 13 & 14 & 15 & 15 & 17 & 18 & 19 \\ 
   \hline
\end{tabular}
\caption{Number of flagged observations by the GY-filter and the HS-filter for increasing number of added outliers placed at $N_2((-0.2,-0.25), 0.01 \mat{I}_2)$.}
\label{tab:n0-sn}  
\end{table}

\begin{figure}[ht]
\begin{center}
\includegraphics[width=0.7\textwidth]{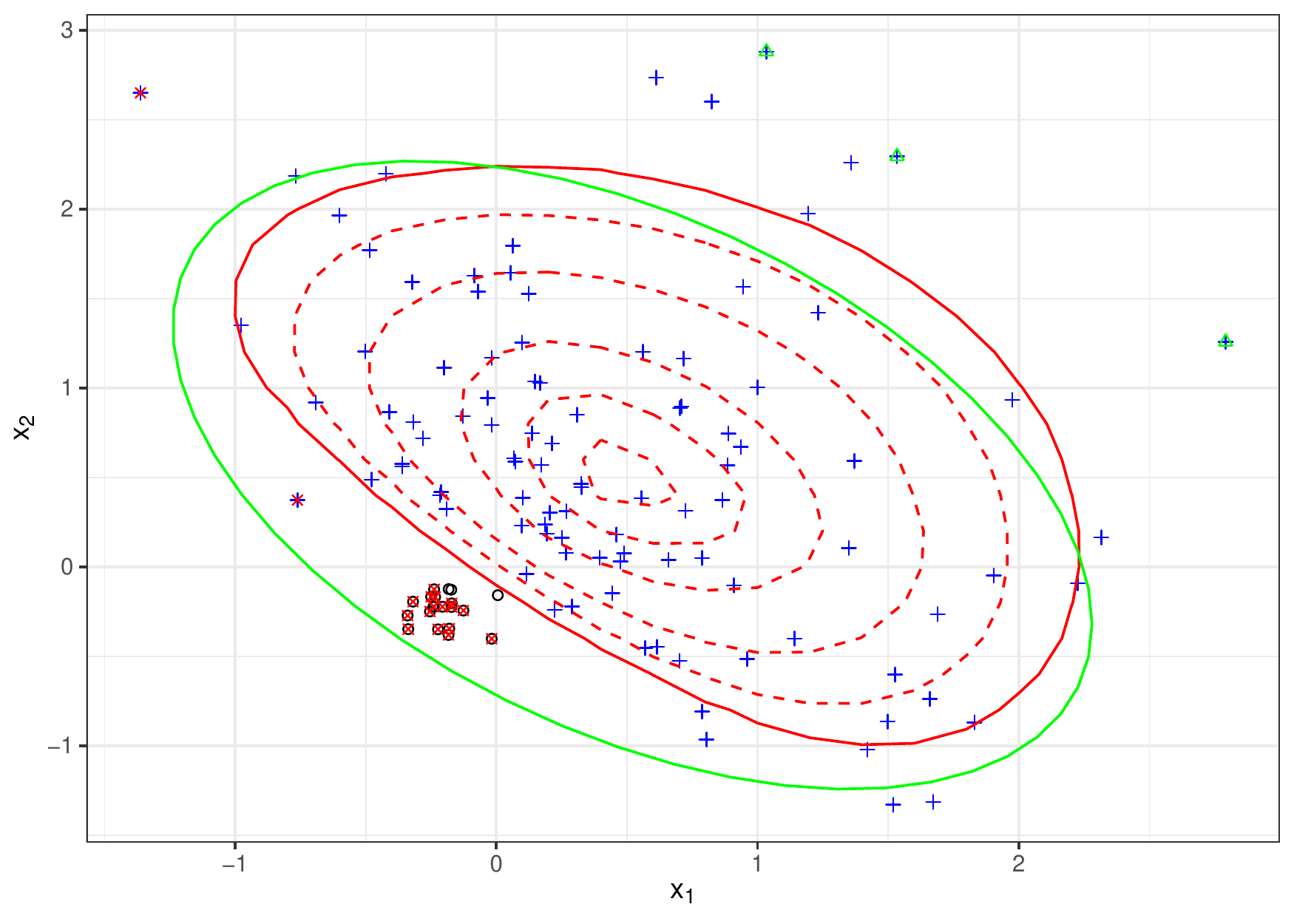}
\end{center}
\caption{$C^\beta(F)$ based on GY-filter is in solid green while for HS-filter is in solid red. Half-space depth contours are red dashed lines and sample observations are blue crosses. The $20$ added outliers are black circles. Observations flagged by the GY-filter are green triangles, while those flagged by the HS-filter are red crosses. Outliers are placed at $N_2((-0.2,-0.25), 0.01 \mat{I}_2)$.} 
\label{fig:outliers-sn}
\end{figure}

In this simulation, we are also interested in identifying such flagged points. Figure \ref{fig:outliers-sn} shows the added outliers at the final iteration (as black circles). Observations flagged by the HS-filter are red crosses, while those flagged by the GY-filter are green triangles. The HS-filter correctly identify the majority of the added cells, while these are never detected by the GY-filter. Indeed, GY-filter flags regular observations which lead to a more symmetric empirical distribution.

In a second experiment we sampled the added outliers from $N_2((-0.5,-0.6), 0.01 \mat{I}_2)$, so that, with high probability, the outliers lie in a region outside the boundary set given by the Mahalanobis distance. While in this case the GY-filter flags the right amount of observations (see Table \ref{tab:n0-sn-2}) most of them do not belong to the set of added outliers. The only effect is, again, to reduce the asymmetry of the observed empirical distribution. Figure \ref{fig:outliers-sn-2} shows the flagged observations after $10$ added outliers (left panel) and at the final step (right panel).

\begin{table}[h]
\centering
\begin{tabular}{rrrrrrrrrrr}
  \hline
$n^o$ of outliers & 1 & 2 & 3 & 4 & 5 & 6 & 7 & 8 & 9 & 10  \\ 
  \hline
  GY-filter & 4 & 5 & 5 & 6 & 6 & 7 & 8 & 9 & 10 & 10 \\ 
  HS-filter & 3 & 3 & 3 & 4 & 5 & 6 & 6 & 7 & 8 & 8 \\ 
   \hline
$n^o$ of outliers & 11 & 12 & 13 & 14 & 15 & 16 & 17 & 18 & 19 & 20 \\ 
  \hline
  GY-filter & 11 & 12 & 13 & 14 & 15 & 16 & 17 & 18 & 19 & 20  \\ 
  HS-filter & 9 & 10 & 11 & 12 & 13 & 13 & 14 & 15 & 16 & 17  \\ 
   \hline
\end{tabular}
\caption{Number of flagged observations by the GY-filter and the HS-filter for increasing number of added outliers placed at $N_2((-0.5,-0.6), 0.01 \mat{I}_2)$.}
\label{tab:n0-sn-2}  
\end{table}

\begin{figure}[htbp]
\begin{center}
  \includegraphics[width=0.49\textwidth]{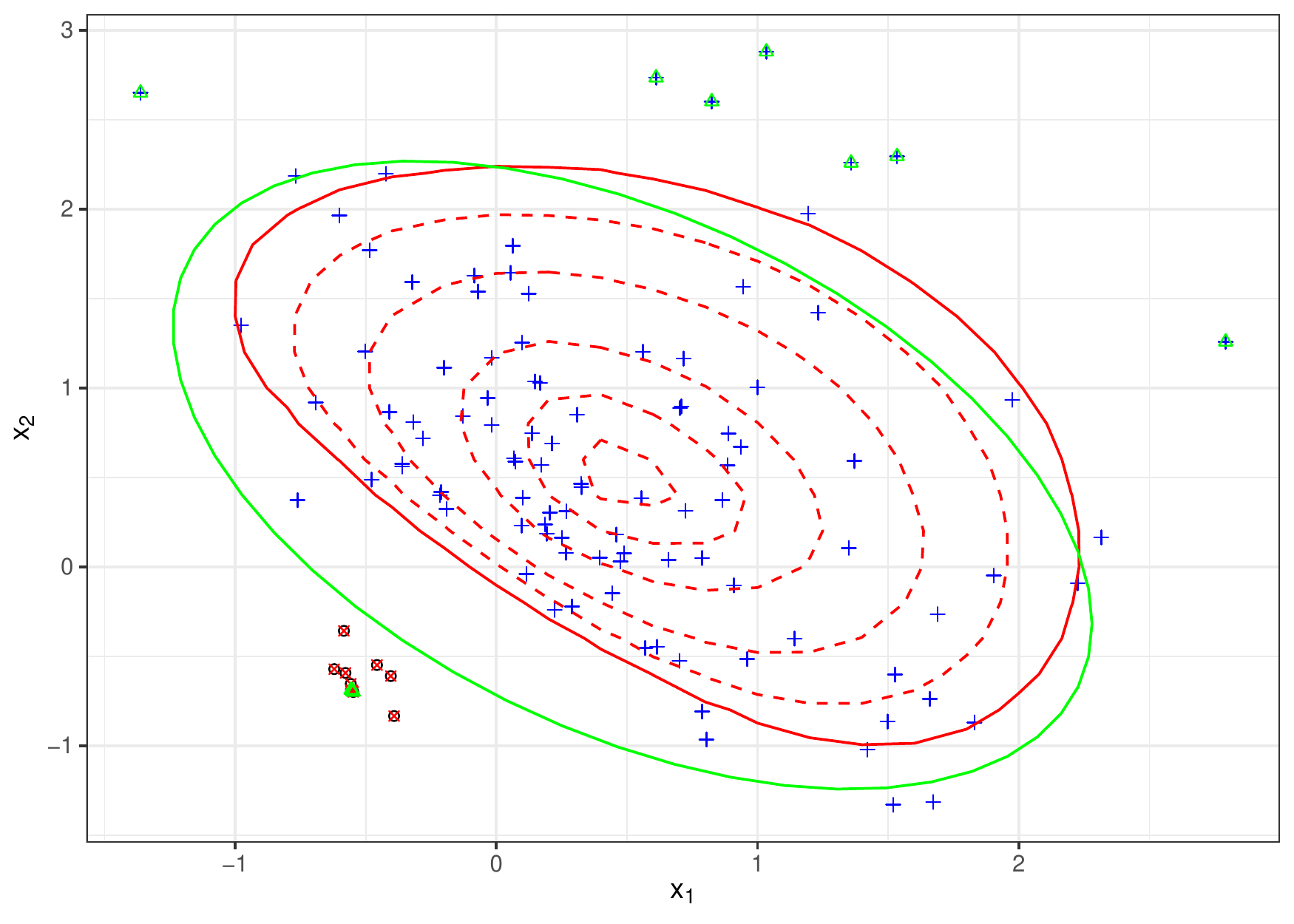}
  \includegraphics[width=0.49\textwidth]{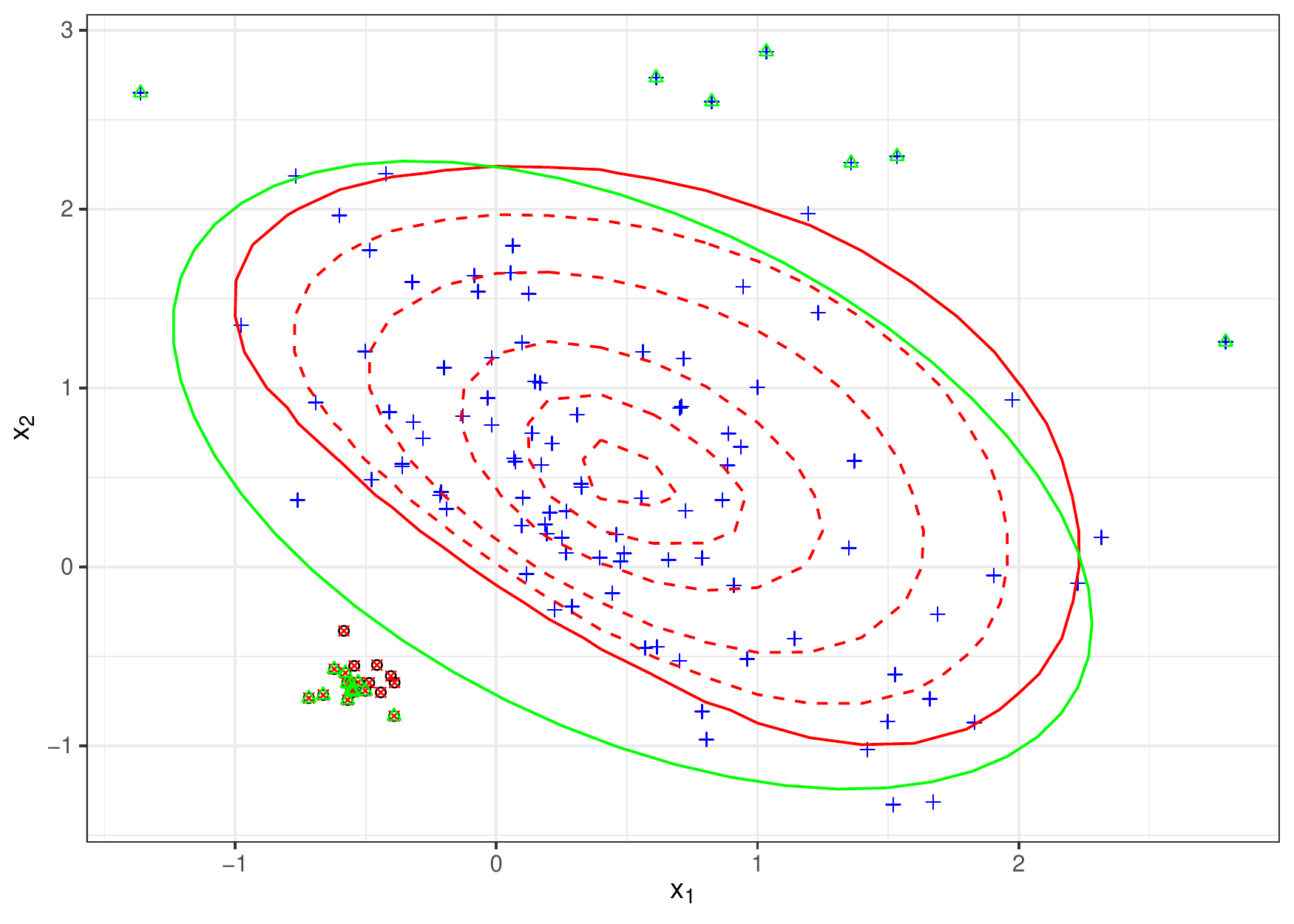}
\end{center}
\caption{$C^\beta(F)$ based on GY-filter is in solid green while for HS-filter is in solid red. Half-space depth contours are red dashed lines and sample observations are blue crosses. The added outliers are black circles. Observations flagged by the GY-filter are green triangles, while those flagged by the HS-filter are red crosses. Outliers are placed at $N_2((-0.5,-0.6), 0.01 \mat{I}_2)$. Left panel: $10$ added outliers, right panel: $20$ added outliers.} 
\label{fig:outliers-sn-2}
\end{figure}

\subsection{Small-cap Stock Returns}
\label{subsec:example-small-cap}

We consider the weekly returns from $01/01/2008$ to $12/28/2010$ for a portfolio of 20 small-cap stocks from \citet{Martin2013}. The data set is publicly available at the link "http://www.bearcave.com/finance/smallcap weekly.csv" and can be found in the \texttt{R} package \texttt{GSEdepth}. Subsection SM--6-2 of the Supplementary Material provides the necessary code to replicate the results and the figures.

\begin{figure}[htb!]
\begin{center}
\includegraphics[width=\textwidth]{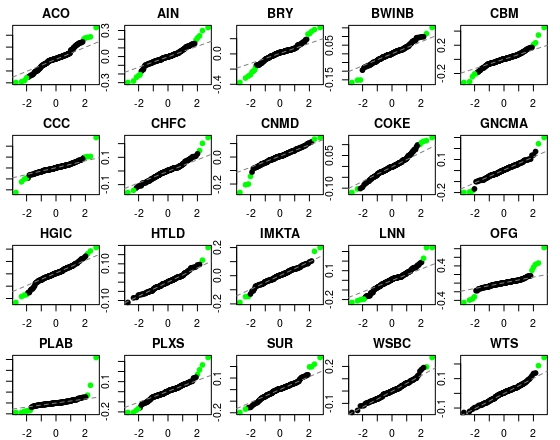}
\end{center}
\caption{Small-cap stock returns. QQ-plots of the variables, green: observations marked as outliers}
\label{fig:scsr-qqplots}
\end{figure}

With this example we want to compare the filter introduced in \citet{Agostinelli2015a} and the same filter with the improvements proposed in \citet{Zamar2017} to the presented filter based on statistical data depth functions obtained using the half-space depth .   

Figure \ref{fig:scsr-qqplots} shows the normal QQ-plots of the 20 variables. The returns in all stocks seem to roughly follow a normal distribution, but with the presence of large outliers. The returns in each stock that lie 3 MAD's away from the coordinate-wise median are displayed in green in the figure. These indicated cells, which are considered cell-wise outliers, correspond to the $4.4\%$ of the total cells and they propagate to $37.6\%$ of the cases.

\begin{figure}
\begin{center}
\includegraphics[width=\textwidth]{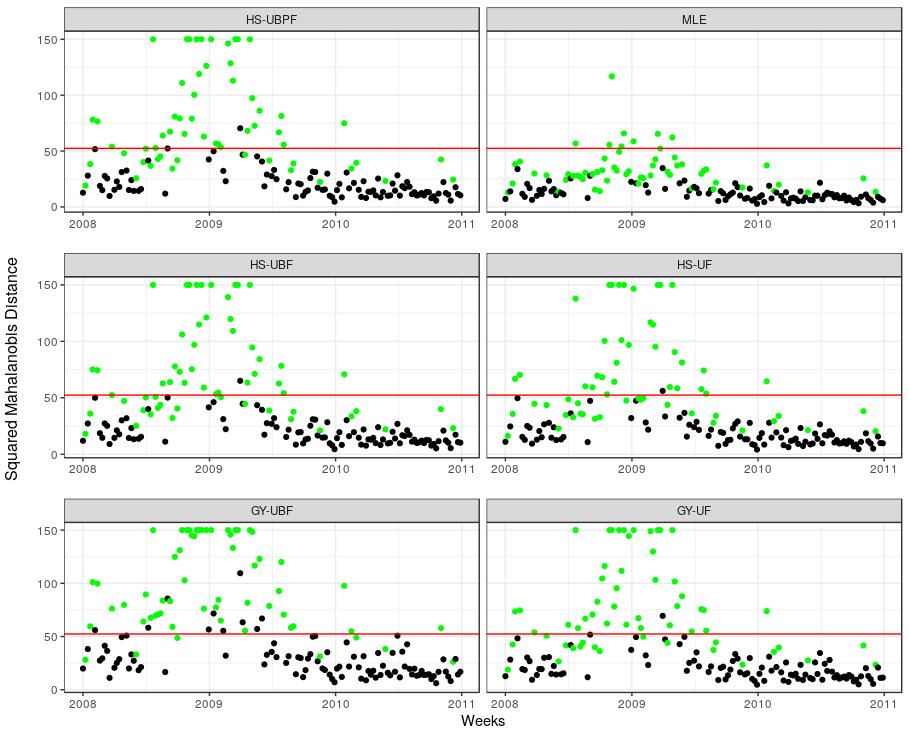}
\end{center}
\caption{Squared Mahalanobis distances of the weekly returns based on the MLE, the GY filters (GY-UF, GY-UBF) and the filters based on half-space depth (HS-UF, HS-UBF, HS-UBPF). Observations with one or more cells flagged as outliers are displayed in green. Large Mahalanobis distance are truncated for a better visualization.}
\label{fig:md}
\end{figure}

Figure \ref{fig:md} shows the squared Mahalanobis distances (MDs) of the weekly returns based on the estimates given by the MLE, the GY-UF, the GY-UBF, the HS-UF, the HS-UBF and the HS-UBPF. Observations with one or more cells flagged as outliers are displayed in green. We say that the estimate identifies an outlier correctly if the MD exceeds the $99.99\%$ quantile of a chi-squared distribution with 20 degrees of freedom. We see that the MLE estimate does a very poor job recognizing only 8 of the 59 cases. The GY-UF, HS-UF, HS-UBF and HS-UBPF show a quite similar behavior, doing better then the MLE but they miss about one third of the cases. The GY-UBF identifies all but seven of the cases. 

Figure \ref{fig:md-ddc} shows the Mahalanobis distances produced by GY-UBF-DDC-C and HS-UBPF-DDC-C. Here, we can see that the GY-UBF-DDC-C misses 13 of 59 cases while the HS-UBPF-DDC-C has missed 12 cases. Although they seem not to do a better job, these two filters are able to flag some other observations, not identified before, as case-wise outliers. 
\begin{figure}
\begin{center}
\includegraphics[width=\textwidth]{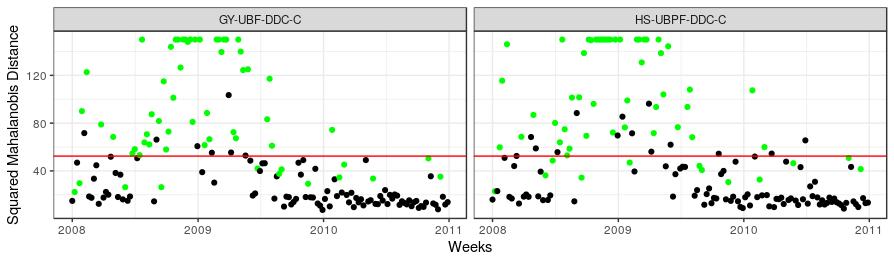}
\end{center}
\caption{Squared Mahalanobis distances of the weekly returns based on the GY-UBF-DDC-C and the corresponding filter based on half-space depth, HS-UBPF-DDC-C). Observations with one or more cells flagged as outliers are displayed in green}
\label{fig:md-ddc}
\end{figure}
\begin{figure}
\begin{center}
\includegraphics[width=0.32\textwidth]{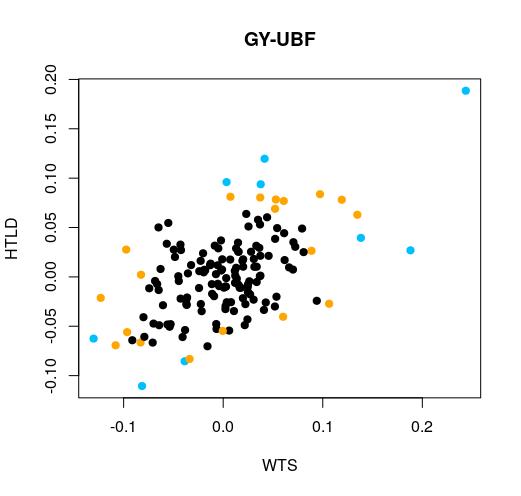}
\includegraphics[width=0.32\textwidth]{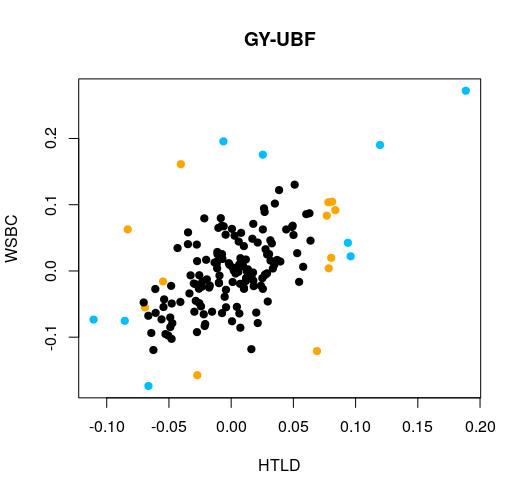}
\includegraphics[width=0.32\textwidth]{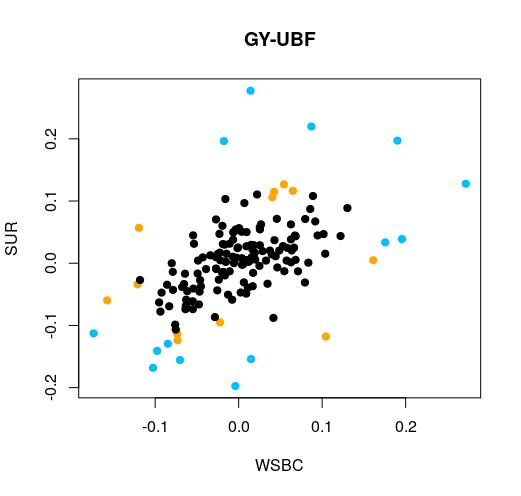} \\
\includegraphics[width=0.32\textwidth]{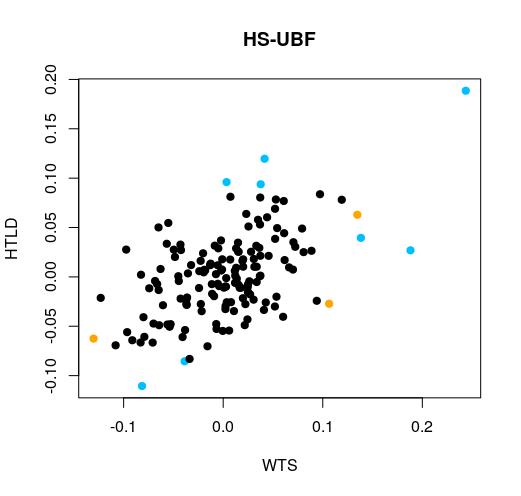}
\includegraphics[width=0.32\textwidth]{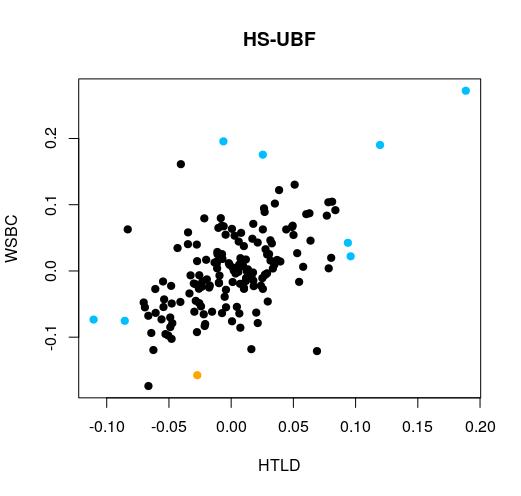}
\includegraphics[width=0.32\textwidth]{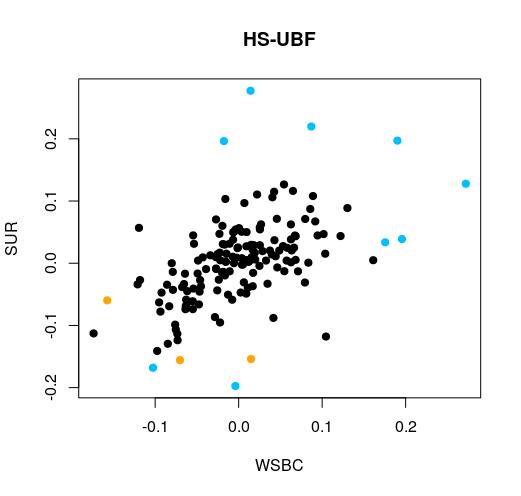}

\end{center}
\caption{Bivariate scatter plot of small-cap stock returns. In the first row the Gervini-Yohai depth is used. Blue: outliers detected by the GY-UF univariate filter; orange: outliers detected by the bivariate step of GY-UBF but not in the univariate step. In the second row the Half-space depth is used. Blue: outliers detected by the HS-UF univariate filter; orange: outliers detected by the bivariate step of HS-UBF but not in the univariate step.}  
\label{fig:scsr}
\end{figure}

Figure \ref{fig:scsr} shows the bivariate scatter plot of WTS versus HTLD, HTLD versus WSBC and WSBC versus SUR where the GY-UBF and HS-UBF filters are applied, respectively. The bivariate observations with at least one component flagged as outlier are in blue, while outliers detected by the bivariate filter, but excluded by the univariate filter, are in orange. We see that the HS-UBF identifies less outliers with respect to the GY-UBF.

\section{Conclusions}
\label{sec:conclusions}

We presented a general idea to construct filters based on statistical data depth functions, called depth-filters. We also showed that previously defined filters can be derived from our general method. We developed one filter, belonging to the family of depth-filters, using the half-space depth, namely HS-filter. Furthermore, our filter is very versatile since it is defined in general dimension $d$, $1 \le d \le p$. Indeed, considering the idea of an univariate and univariate-bivariate filter, we applied our HS-filter using both $d=1$ and $d=2$, and we proposed a new filtering procedure adding the case $d=p$, in sequence. Finally, we combined the depth-filter HS-UBPF and DDC, as suggested by \cite{Zamar2017}. After the filtering process, the generalized S-estimator was applied, following the two-step procedure introduced in \citet{Agostinelli2015a}.

The results of the simulation study show that GY-UBF and HS-UBPF, combined with DDC, outperform the other filters in the case-wise contamination scenario. However, for small $p$, HS-UBPF outdoes the other filters, even if its computational time could slightly increase, in both case-wise and cell-wise contamination and improves for increasing $n$. Finally, it is not suggested to combine any filter with DDC if cell-wise outliers are present, indeed, even if GY-UBF-DDC-C and HS-UBPF-DDC-C may show lower maximum average LRT and average MSE values, they do not have the best behaviour with respect different contamination values $k$.

Further research on this filter could be needed to explore the performance of the estimator in different types of data, for example in flat data sets (e.g., $n \approx 2p$). In addition, different statistical data depth functions could be used in place of the half-space depth to construct new filters. The choice of the appropriate statistical data depth function could be helpful to analyze different types of data.

\bibliography{depth}

\end{document}